\documentclass[reqno,11pt]{amsart}
\usepackage{amsmath,amssymb,mathrsfs,amsthm,amsfonts}
\usepackage[inline]{enumitem}
\usepackage[usenames,dvipsnames]{xcolor}
\usepackage{hyperref}
\usepackage[percent]{overpic}
\usepackage{comment}
\usepackage{stmaryrd}
\usepackage{algorithm}
\usepackage{algorithmic}
\usepackage{todonotes}
\usepackage{mathtools}
\usepackage{subcaption}
\usepackage{tikz}
\usepackage{bm}
\usetikzlibrary{calc}
\usetikzlibrary{arrows.meta,backgrounds}
\usepackage{multirow,array,longtable,booktabs}
\usetikzlibrary{arrows}

\hypersetup{
	colorlinks=true, linkcolor=blue,
	citecolor=ForestGreen
}
\usepackage[paper=letterpaper,margin=1in]{geometry}
\DeclareMathAlphabet{\mathpzc}{OT1}{pzc}{m}{it}

\newtheorem{theorem}{Theorem}[section]
\newtheorem{lemma}[theorem]{Lemma}
\newtheorem{proposition}[theorem]{Proposition}

\newtheorem{corollary}[theorem]{Corollary}
\theoremstyle{definition}
\newtheorem{definition}[theorem]{Definition}

\newtheorem{remark}[theorem]{Remark}
\numberwithin{equation}{section}
\allowdisplaybreaks
\usepackage{acronym}

\acrodef{LDP}{Large Deviation Principle}

\newcommand{\red}{\textcolor{red}}

\newcommand{\be}{\begin{equation}}
\newcommand{\ee}{\end{equation}}
\newcommand{\bea} {\begin{array}{rl}}
\newcommand{\eea} {\end{array}}
\newcommand{\bepa}{\left\{ \begin{array}{l}}
\newcommand{\eepa} {\end{array}\right.}
\makeatletter

\newcommand{\R}{\mathbb{R}} 

\newcommand{\Z}{\mathbb{Z}} 

\newcommand{\e}{\varepsilon}

\newcommand{\T}{\mathbb{T}}

\renewcommand{\bar}{\overline}

\usepackage{graphicx}

\newcommand{\bx}{\bm{x}}
\newcommand{\bX}{\bm{X}}

\title{Singular Perturbations of Hamilton-Jacobi equations in the Wasserstein space}

\author[A.\ Zitridis]{Antonios Zitridis}
\address{A.\ Zitridis,
	Department of Mathematics, University of Chicago,
	\newline\hphantom{\quad \ \ A. Zitridis}
	5734 S.~University Avenue, Chicago, Illinois 60637 USA
}
\email{zitridisa@uchicago.edu}

\begin{document}       
	\begin{abstract}
We study a singular perturbation problem for second-order Hamilton-Jacobi equations in the Wasserstein space. Specifically, we characterize the behavior of the solutions as the perturbation parameter $\e$ tends to zero. The notion of solution we adopt is that of viscosity solutions in the sense of test functions on the Wasserstein space. Our proof utilizes the perturbed test function method, appropriately adapted to this setting. Finally, we highlight a connection with the homogenization of conditional slow-fast McKean-Vlasov stochastic differential equations.
	\end{abstract}

	
	\maketitle
	{
		}
 \vspace{-8mm}
\section{Introduction}

\noindent
\subsection{Background}
Let $\e, T>0$, $d\in \mathbb{N}^*$, $A$ a compact subset of $\R^{2d}$ and $(B^k)_{k=0,1,2}$ independent $d$-dimensional Brownian motions defined on a fixed filtered probability space $(\Omega,\mathcal{F},\mathbb{F}=(\mathcal{F}_t)_{t\ge 0},\mathbb{P})$. We consider the system of controlled conditional slow-fast McKean-Vlasov stochastic differential equations (SDEs for short) on the $d$-dimensional torus $\T^d=\R^d/\Z^d$
\be\label{eq:slow-fastMcKeanVlasov}
\begin{cases}
dX^{\e}_t &=  c(X^{\e}_t,Y^{\e}_t,\alpha_t,\mathcal{L}^0(X_t^\e))dt + \sqrt{2}\sigma_1 dB^1_t+\sqrt{2}\sigma_{0,1}dB^0_t,\\
dY^{\e}_t & = \frac{1}{\e}f(X^{\e}_t,Y^{\e}_t,\alpha_t,\mathcal{L}^0(X_t^{\e}))dt + \frac{\sqrt{2}}{\sqrt{\e}}\sigma_2 dB^2_t+\frac{\sqrt{2}}{\sqrt{\e}}\sigma_{0,2}dB^0_t,
\end{cases}
\ee
where $(\mathcal{F}_t^{0})_{0\le t\le T}$ is the filtration generated by $B^0$, $\mathcal{L}^0(X_t)$ is the law of $X_t$ conditioned upon $\mathcal{F}_t^0$, $\alpha_t$ is an adapted process with values in $A$ and where $f,c:\T^{d}\times\T^d\times  A\rightarrow \R^d$, $\sigma_1,\sigma_2,\sigma_{0,1},\sigma_{0,2}$ are $d\times d$ matrices. Here $B^1,B^2$ play the role of the idiosyncratic noise, while $B^0$ is the common noise.
\vspace{1mm}

\noindent
We also introduce the functions $\mathcal{G}:\mathcal{P}(\T^d)\rightarrow \R$ \footnote{We use the notation $\mathcal{P}(X)$ for the Wasserstein space, that is the space of probability measures over a Polish space $X$ equipped with the Wasserstein distance. We refer to Appendix \ref{appendix A}.}, $L:\T^{d}\times\T^d\times A\times\mathcal{P}(\T^d)\rightarrow \R$ and we let $(t_0,m_0)\in [0,T]\times\mathcal{P}(\T^{2d})$. The aim of this paper is to study the behavior as $\e\rightarrow 0$ of the mean field control problem
\begin{align} 
U^{\e}(t_0,m_0):=\inf_{\alpha\in \mathcal{A}}\mathbb{E}\bigg[\int _{t_0}^TL&\left(X_t^{t_0,\e,m_0},Y_t^{t_0,\e,m_0},\alpha_t,\mathcal{L}^{0}(X_t^{t_0,\e,m_0})\right)dt
+\mathcal{G}\left(\mathcal{L}^{0}(X_T^{t_0,\e,m_0})\right)\bigg], \label{optimalc}
\end{align}
where the infimum is taken over $\mathcal{A}$ the set of all controls, i.e. the set of all $\mathbb{F}$-adapted processes $\alpha$ with values in the compact set $A$ and $(X_t^{t_0,\e,m_0},Y_t^{t_0,\e,m_0})$ satisfies \eqref{eq:slow-fastMcKeanVlasov} with initial condition on the law $\mathcal{L}((X_{t_0}^{t_0,\e,m_0}, Y_{t_0}^{t_0,\e,m_0}))=m_0\in \mathcal{P}(\T^{2d})$.
\vspace{2mm}

\noindent
The motivation of our problem comes from the study of the behavior of slow-fast systems such as \eqref{eq:slow-fastMcKeanVlasov} (as $\e\rightarrow 0$) in the absence of controls ($c,f$ do not depend on $\alpha$). This is a topic of intense research and we cite the initial works \cite{pardoux2001poisson,Pardoux2}, which were later generalized in \cite{rockner2021diffusion, rockner2021strong, li2022poisson, qiao2022efficient, hong2025diffusion,bezemek2023rate,hou2024asymptotic}, in the case without common noise ($\sigma_{0,1}=\sigma_{0,2}=0$) and when the coefficients depend on the law of the slow variable. The case of non-zero $\sigma_{0,1},\sigma_{0,2}$ was addressed in \cite{zitridis2023homogenization}. Similar problems were also studied in \cite{bardi2023singular, bensoussan2005singular, kumar2012singular}, in different settings and for initial conditions in \eqref{eq:slow-fastMcKeanVlasov} in $\R^{2d}$, that is $(X_{t_0}^\e,Y_{t_0}^\e)=(x,y)\in\R^{2d}$. Listing all the related literature is beyond the scope of this paper.
\vspace{1mm}

\noindent
Staying within the framework where there are no controls, a crucial tool in analyzing the behavior of $\mathcal{L}(X_t^{\e})$ in \cite{bezemek2023rate,zitridis2023homogenization} is the partial differential equation (PDE) satisfied by $U^{\e}$ from \eqref{optimalc} with $L=0$, which is $\mathcal{G}\left( \mathcal{L}(X^{t_0,\e,m_0}_T)\right)=:V^{\e}(t_0,m_0)$. This is a linear PDE over the Wasserstein space $\mathcal{P}(\T^{2d})$ and, under suitable regularity assumptions on $c,f,\mathcal{G}$, it was shown that $V^{\e}$ is its unique classical solution. In \cite{zitridis2023homogenization}, this enabled the derivation of the asymptotic behavior of $V^{\e}$ as $\e \to 0$ via a perturbation argument.
\vspace{2mm}

\noindent
In the case where $c,f$ in \eqref{eq:slow-fastMcKeanVlasov} depend on the controls $\alpha\in \mathcal{A}$, the PDE satisfied by $U^{\e}$ is no longer linear and is the Hamilotn-Jacobi equation posed in $[0,T]\times \mathcal{P}(\mathbb{T}^{2d})$ \footnote{The precise definition of  solution is provided in the next section. $D_mU$ denotes the Wasserstein derivative; we refer to Appendix A.}
\be \label{HJ1}
\begin{cases}
    -\partial_t U^{\e} +\int_{\mathbb{T}^{2d}}H^{\e}\left(x,y,D_mU^{\e}(t,m,x,y),\pi^1_*m\right)m(dx,dy)\\
    \hspace{1cm}-\int_{\mathbb{T}^{2d}}\text{tr}\left((\sigma^{\e}\sigma^{\e\top}+\sigma_0^{\e}\sigma_0^{\e\top})D_{(x,y)}D_mU^{\e}(t,m,x,y)\right)m(dx,dy)\\
    \hspace{1.2cm} -\int_{\T^{2d}}\int_{\T^{2d}}\text{tr}\left( \sigma_0^{\e}\sigma_0^{\e\top}D^2_{mm}U^{\e}(t,m,x,y,x',y')\right)m(dx,dy)m(dx',dy')=0,\\
    U^{\e}(T,m)=\mathcal{G}(\pi^1_*m),
\end{cases}
\ee
where for $m\in\mathcal{P}(\T^{2d})$, $\pi^1_*m\in \mathcal{P}(\mathbb{T}^d)$ is the the $x-$marginal of $m$ \footnote{This means that $\pi^1_*m(A)=m(A\times\T^d)$, for any measurable $A\subset \T^d$.}, $H:\mathbb{T}^d\times \mathbb{T}^d\times\R^{d}\times\R^d\times \mathcal{P}(\T^d)\rightarrow \R$ is the Hamiltonian \footnote{We denote by $a\cdot b$ the dot product of $a,b\in \R^d$.}
\be\label{Hamiltonian}
H(x,y,p,q,\mu)=\sup_{\alpha\in A}\left\{ -p\cdot c(x,y,\alpha,\mu)-q\cdot f(x,y,\alpha,\mu) -L(x,y,\alpha,\mu)\right\}
\ee 
and, for $x,y\in \mathbb{T}^d$, $p,q\in \R^{d}, \mu\in\mathcal{P}(\T^d)$
\begin{align*}
\hspace{1cm}H^\e(x,y,p,q,\mu):=H\left(x,y,p,\frac{q}{\e},\mu\right),\;
\sigma^{\e}:=\begin{pmatrix}
\sigma_1 & 0 \\
 0 & \frac{1}{\sqrt{\e}}\sigma_2
\end{pmatrix}
\text{ and  }\sigma_0^{\e}:=\begin{pmatrix}
\sigma_{0,1} & 0 \\
 \frac{1}{\sqrt{\e}}\sigma_{0,2} & 0
\end{pmatrix},\nonumber
\end{align*}
with $0$ 
being the $d\times d$ zero matrix.
\vspace{2mm}

\noindent
In general, Hamilton-Jacobi equations in infinite dimensions, such as \eqref{HJ1}, do not have classical solutions even if the data is smooth (see \cite{briani2018stable} and \cite[Example 2.4]{daudin2024comparison} for examples). As a consequence, it is essential to work with weaker notions of solutions such as monotone solutions or various notions of viscosity solutions, depending on the setup and the assumptions. For results in these directions, we refer to \cite{bertucci2021monotone, bertucci2024lipschitz, cardaliaguet2021weak} in the context of mean field games and \cite{cosso2024master, soner2024viscosity, bayraktar2023comparison,cheung2025viscosity,bayraktar2025viscosity} for mean field control, among others. 
\vspace{2mm}

\noindent
In this paper, we consider viscosity solutions of \eqref{HJ1} in the sense of test functions (see Definition \ref{vdef}). In particular, it was proved in \cite{daudin2023well} that there is a unique viscosity solution $U^{\e}$ of \eqref{HJ1} (we refer to section \ref{viscosity}). By adopting this definition, our goal is to study the behavior of $U^{\e}$ as $\e\rightarrow 0$. 

\subsection{Main results} Our main result is a homogenization result for \eqref{HJ1} and can be stated as follows

\begin{theorem}\label{main}
Assume \textbf{(A)} from section \ref{assu} is in place. Then, there exists a function $\overline{H}:\mathbb{T}^d\times \R^d\times\mathcal{P}(\T^d)\rightarrow \R$ such that for $(x,p,\mu)\in \T^d\times\R^d\times\mathcal{P}(\T^d)$, $\overline{H}(x,p,\mu)$ is the unique constant such that 
$$-\text{tr}\left((\sigma_2\sigma_2^{\top}+\sigma_{0,2}\sigma_{0,2}^\top )D^2_{yy}v\right)+H(x,y,p,D_yv,\mu)=\overline{H}(x,p,\mu),\;\;y\in \T^d,$$
has a unique, up to a constant, classical solution $v(y)=v(y;x,p,\mu)$ and the equation on $[0,T]\times \mathcal{P}(\mathbb{T}^d)$
\be \label{HJ2}
\begin{cases}
    -\partial_t U +\int_{\mathbb{T}^{d}}\overline{H}(x,D_{\mu}U(t,\mu,x),\mu)\mu(dx)\\
    \hspace{2cm}-\int_{\mathbb{T}^{d}}\text{tr}\left((\sigma_1\sigma_1^{\top}+\sigma_{0,1}\sigma_{0,1}^{\top})D_{x}D_{\mu}U(t,\mu,x)\right)\mu(dx)\\
    \hspace{3cm}-\int_{\T^{d}}\int_{\T^d}\text{tr}\left(\sigma_{0,1}\sigma_{0,1}^\top D^2_{\mu\mu}U(t,\mu,x,y)\right)\mu(dx)\mu(dy)=0,\\
    U(T,\mu)=\mathcal{G}(\mu)
\end{cases}\ee
has a unique viscosity solution $U: [0,T]\times \mathcal{P}(\mathbb{T}^d)\rightarrow \R$. Furthermore, $U^{\e}(t_0,m_0)$ converges as $\e\rightarrow 0$ to $U(t_0,\pi^1_*m_0)$ uniformly in $[0,T]\times\mathcal{P}(\mathbb{T}^{2d})$.
\end{theorem}
\vspace{1mm}

\begin{remark}
In the literature the function $\overline{H}(x,p,\mu)$ is called the effective Hamiltonian associated with the Hamiltonian $H(x,y,p,q,\mu)$. Therefore, the function $\overline{\mathcal{H}}(\mu,p)=\int_{\T^d}\overline{H}(x,p,\mu)\mu(dx)$, $(\mu,p)\in\mathcal{P}(\T^d)\times \R^d$ serves as the effective Hamiltonian associated to the Hamiltonian of \eqref{HJ1}.
\end{remark}
\vspace{1mm}

\noindent
Theorem \ref{main} extends an already known result for Hamilton-Jacobi equations in finite dimensions. In particular, if $\sigma_0^{\e}$ is the zero $2d\times 2d$ matrix, the finite dimensional analogue of \eqref{HJ1} is a Hamilton-Jacobi equation posed in $\T^{2d}$ instead of $\mathcal{P}(\T^{2d})$ and takes the form 
\be\label{simpler}
\begin{cases}
-\partial_t u^{\e}+h\left(x,y,D_xv^{\e},\frac{1}{\e} D_yu^{\e}\right)-\text{tr}\left(\sigma_1\sigma_1^{\top}D^2_{xx}u^{\e}\right) -\frac{1}{\e}\text{tr}\left(\sigma_2\sigma_2^{\top}D^2_{yy}u^{\e}\right)=0,&\\
u^{\e}(T,x,y)=g(x),
\end{cases}\ee
for a given Hamiltonian $h:\T^{2d}\times\R^{2d}\rightarrow \R$ and some function $g:\T^d\rightarrow \R$. This problem was studied in \cite{bensoussan2005singular} and, for fully nonlinear equations, in \cite{alvarez2003singular} and the references therein. The main observation to study the behavior of $u^{\e}$ as $\e\rightarrow 0$ is that the family of viscosity solutions of \eqref{simpler} $(u^\e)_{\e>0}$ is uniformly bounded, consequently we may consider the semi-limits 
\begin{align*}
&\overline{u}(t,x):=\limsup_{\e\rightarrow 0, (t',x')\rightarrow (t,x)}\sup_{y\in \T^d} u^{\e}(t',x',y),\text{ if }t<T,\\
&\overline{u}(T,x):=\limsup_{(t',x')\rightarrow (T,x), t'<T}\overline{u}(t',x')
\end{align*}
and $\underline{u}$ (defined as $\overline{u}$ but with $\liminf,\; \inf$ instead $\limsup,\; \sup$, respectively). It is clear that 
$\overline{u}\geq \underline{u}$. On the other hand, the semi-limits $\overline{u},\underline{u}$ can be proved to be a subsolution and a supersolution of the equation
\be\label{effective}
\begin{cases}
   -\partial_tu +\overline{h}(x,D_xu)-\text{tr}\left(\sigma_1\sigma_1^{\top}D^2_{xx}u\right)=0,\\
   u(T,x)=g(x),
\end{cases}\ee
respectively, where $\overline{h}$ is another function called the effective Hamiltonian. The effective Hamiltonian has the property that for $(x,p)\in \T^d\times\R^d$, the equation 
$$-\text{tr}\left( \sigma_2\sigma_2^{\top}D^2_{yy}v\right)+h(x,y,p,D_yv)=\overline{h}(x,p),\;\;\;y\in \T^d$$
has a unique viscosity solution. By comparison principle, it is expected that $\overline{u}\leq \underline{u}$, hence $\overline{u}=\underline{u}:=u$, which will be a viscosity solution of \eqref{effective}. The proof of Theorem \eqref{main} relies on this idea, properly adapted in the viscosity solutions setting for Hamilton-Jacobi equations in infinite dimensions.
\vspace{2mm}

\noindent
If we restrict Theorem \ref{main} to the case where $L=0$ and the functions $f,c$ do not depend on $\alpha$, we may obtain a homogenization result for slow-fast systems of SDEs. We state it as a Corollary.

\begin{corollary}\label{linear case}
    Assume $\textbf{(A)}$ from section 2 in the particular case where $L=0,\; \mathcal{F}=0$ and $c,f$ are independent of $\alpha$. We also consider $T\ge t_0\ge 0$ and $(X_t^{t_0,\e,m_0},Y_t^{t_0.\e,m_0})$ to be a solution of \eqref{eq:slow-fastMcKeanVlasov} with initial condition on the law $\mathcal{L}(X_{t_0}^{t_0,\e,m_0},Y_{t_0}^{t_0.\e,m_0})=m_0$. Then, $\mathcal{G}\left(\mathcal{L}(X_T^{t_0,\e,m_0})\right)\xrightarrow{\e\rightarrow 0}\mathcal{G}\left(\mathcal{L}(X_T^{t_0,\pi^1_*m_0})\right)$, where $X_t^{t_0,\pi^1_*m_0}$ is the process satisfying the SDE
    \be\label{SDE cor}
        dX_t=\overline{c}(X_t,\mathcal{L}^0(X_t))dt+\sqrt{2}\sigma_1dB'_t+\sqrt{2}\sigma_{0,1}dB^0_t,\;\;\mathcal{L}(X_{t_0})=\pi^1_*m_0,
    \ee
    $B'$ is a Brownian motion on a possibly different probability space and $\overline{c}:\T^d\times \mathcal{P}(\T^d)\rightarrow \R^d$ is a function that depends on $c, f,\sigma_{2}$ and $\sigma_{0,2}$.
\end{corollary}

\begin{remark}
Since $\mathcal{G}$ is only required to be Lipschitz continuous, Corollary \ref{linear case} relaxes the conditions of previous homogenization results for slow-fast systems of McKean-Vlasov SDEs defined on the torus $\T^d$ in the case where $\sigma_{0,1},\sigma_{0,2}$ are not state dependent; e.g. \cite[Remark 2.3(vii), Theorem 2.1]{zitridis2023homogenization} and \cite[Theorem 3.1]{bezemek2023rate} (on the torus and if $\sigma_{0,1}=\sigma_{0,2}=0$).  
\end{remark}

\vspace{1mm}

\subsection{Extensions and Outlook}
Our proof of Theorem \ref{main} can be extended to the case where \eqref{eq:slow-fastMcKeanVlasov} takes the form
$$
\begin{cases}
dX^{\e}_t &=  c(X^{\e}_t,Y^{\e}_t,\alpha_t,\mathcal{L}^0(X_t^\e))dt + \sqrt{2}\sigma_1(X_t^{\e},\mathcal{L}^0(X_t^{\e})) dB^1_t+\sqrt{2}\sigma_{0,1}dB^0_t,\\
dY^{\e}_t & = \frac{1}{\e}f(X^{\e}_t,Y^{\e}_t,\alpha_t,\mathcal{L}^0(X_t^{\e}))dt + \frac{\sqrt{2}}{\sqrt{\e}}\sigma_2(X_t^{\e},Y_t^{\e},\mathcal{L}^0(X_t^{\e})) dB^2_t+\frac{\sqrt{2}}{\sqrt{\e}}\sigma_{0,2}dB^0_t,
\end{cases}
$$
as long as we include the following assumption (A6) in \textbf{(A)}. We omit this more general case to simplify the presentation.
\vspace{2mm}

(A6) $\sigma_1:\T^d\times\mathcal{P}(\T^d)\rightarrow \R^{d\times d}$ and $\sigma_2:\T^d\times \T^d\times\mathcal{P}(\T^d)\rightarrow \R^{d\times d}$ are Lipschitz continuous functions, such that for every $\mu\in\mathcal{P}(\T^d)$, $(x,y)\mapsto \sigma_2(x,y,\mu)$ is $C^2$.

\vspace{2mm}

\noindent
On the other hand, our proof does not work when $\sigma_{0,1},\sigma_{0,2}$ are state dependent. The main difficulty is the comparison principle for \eqref{HJ2}, which is necessary in order to establish that the upper and lower semi-limits are equal (see proof of Theorem \ref{main}). For a general second order Hamilton-Jacobi equation in the Wasserstein space a comparison principle for viscosity sub/super-solutions has been proved in \cite{bayraktar2025viscosity}, provided that the Hamiltonian has a specific structure. In particular, \eqref{HJ1} will have a unique viscosity solution and the perturbed test function argument would still work if we adopt the definition from \cite{bayraktar2025viscosity}, however $\overline{H}$ has to satisfy the condition:
\be\label{suffcondition}
\overline{H}(x,p)=\sup_{\alpha\in A}\left\{ -b(x,\alpha)\cdot p -L'(x,\alpha)\right\},
\ee
for some $b,L': \T^d\times A\rightarrow \R$ continuous and Lipschitz continuous in $x$. If \eqref{suffcondition} does not hold, it is not known whether there is comparison for \eqref{HJ2}.
\vspace{1mm}

\noindent
The following natural directions, where there is comparison principle for viscosity subsolutions and supersolutions, also remain unanswered:
\begin{itemize}
    \item Fully non-linear and convex in $p,q,X,Y$ Hamiltonians $H(x,y,p,q,X,Y): \T^d\times\T^d\times\R^d\times \R^d\times \R^{d\times d}\times \R^{d\times d}\rightarrow \R$ (see \cite{cosso2024master,cheung2025viscosity}). In this case a different analysis is needed at the level of Proposition \ref{lemma1}, as well as a representation for the effective Hamiltonian analogous to \eqref{suffcondition}.
    \item The function $\mathcal{G}$ in \eqref{HJ1} depends on the full measure $m$ instead of just the first marginal $\pi^1_*m$. The finite dimensional analogue of this would be for $g$ in \eqref{simpler} to depend on both $x,y$ and was studied in \cite{alvarez2003singular} (and references therein).
\end{itemize}
We leave these questions for future research.
\vspace{2mm}

\noindent
We finish the introduction with the organization of the paper

\noindent
\subsection{Organization of the paper.} Section 2 introduces the notation and assumptions used throughout the paper. In Section 3, we provide the definition of viscosity solutions and state some basic properties that are already established in the literature. Section 4 outlines the main ideas of the proof, which is based on the perturbed test function method adapted to the infinite-dimensional setting. In this section, we construct the effective Hamiltonian, establish some of its regularity properties, and prove the main results. Definitions and properties of the Wasserstein space and Wasserstein derivatives, as well as the proofs of certain technical lemmas, are deferred to the appendix.

\section{Notation and Assumptions}\label{assu}

\noindent
We now introduce the notations and assumptions used throughout the rest of the paper.
\vspace{2mm}

\noindent
\textbf{Notation.} Through the note $T>0$. We will be using the symbol $\bx:= (x,y)\in \T^d\times\T^d$. For a Polish space $X$ and $p\geq 1$, we will denote the space of probability measures over $X$ with finite $p-$moment by $\mathcal{P}_p(X)$ and by ${\bf d}_p$ its distance; we refer to Appendix \ref{appendix A} for the precise definitions. We use the notation $\mathcal{P}(X):=\mathcal{P}_1(X)$. Differentiation of functions defined in this space is explained in Appendix \ref{appendix A}. For a function $U:\mathcal{P}(X)\rightarrow \R$, we write $U\in C^{2,2}$ if $D_mU(m,x),\; D_xD_mU(m,x)$ and $D^2_{mm}U(m,x,y)$ exist and are continuous, where $D_mU, D^2_{mm}U$ are the first and second Wasserstein derivatives of $U$, respectively.\\
Let $m\in \mathcal{P}_p(\T^{2d})$ and $\pi^1:\T^{d}\times\T^d\rightarrow \T^d$ with $\pi^1(x,y)=x$. We denote by $\pi^1_*m\in \mathcal{P}_p(\T^d)$ the $x-$marginal of the measure $m$. If $f:\T^d\rightarrow \T^d$ and $\mu\in\mathcal{P}(\T^d)$, then $f_*\mu\in\mathcal{P}(\T^d)$ is the pushforward measure such that $f_*\mu(B)=\mu(f^{-1}(B))$ for any measurable $B\subset \T^d$.\\
We use the symbol $D_{(x,y)}u(x,y)= (D_xu,D_yu)$ for the gradient as a vector of the gradients. $u\in C^k$ means that the function is $k$ times continuously differentiable. For a function $u(t,x)$ we write $u\in C^{1,2}$ if the derivatives $D_tu,\; D_xu, D^2_{xx}u$ exist and are continuous. We also use the notation $u\in C^2_x$ (resp. $C^1_x$) if the derivatives $D_xu$ and $D^2_{xx}u$ (resp. $D_xu$) exist and are continuous with respect to the $x$ variable. We write $u\in C^{2,\beta}$ if $u$ is twice continuously differentiable and the second derivative is $\beta$-H\"older continuous. We use the usual H\"older norm. We denote by $f\ast g(x):=\int f(y)g(x-y)dy$ the convolution between two functions.
\vspace{3mm}

\noindent
\textbf{Assumptions.} We work under the following assumptions on the functions $H$ and $\mathcal{G}$. Below, $A$ is a compact subset of $\R^{2d}$.
\vspace{3mm}

(A1) $c,f: \mathbb{T}^d\times\mathbb{T}^d\times A\times \mathcal{P}(\T^d)\rightarrow \R^d$ are continuous and there exists a constant $K\geq 0$ such that 
\begin{align*}
|c(x,y,\alpha,\mu)-c(x',y',\alpha,\mu')|+|f(x,y,\alpha)-f(x',y',\alpha)|\leq K\left(|x-x'|+|y-y'|+{\bf d}_1(\mu,\mu')\right),
\end{align*}
for all $\alpha\in A$, $x,x',y,y'\in \mathbb{T}^d$ and $\mu,\mu'\in\mathcal{P}(\T^d)$, where $|x-x'|$ denotes the Euclidean norm of $x-x'$.\\
\vspace{1mm}

(A2) The Lagrangian $L: \mathbb{T}^d\times\mathbb{T}^d\times A\rightarrow \R$ is continuous and satisfies
$$|L(x,y,a,\mu)-L(x',y',a,\mu')|\leq C_L\left(|x-x'|+|y-y'|+{\bf d}_1(\mu,\mu')\right),$$
for some constant $C_L$, for all $\alpha \in A$, $x,x',y,y'\in \mathbb{T}^d$ and $\mu,\mu'\in\mathcal{P}(\T^d)$.\\
\vspace{1mm}

(A3) The Hamiltonian $H:\mathbb{T}^{2d}\times  \R^{2d}\times \mathcal{P}(\T^d)\rightarrow \R$ satisfies for some constants $C_{H}>0$ and $\beta\in (0,1)$:

\null\hspace{1cm}(1) $H$ satisfies \eqref{Hamiltonian} with $c,f,L$ as in assumptions (A1) and (A2).

\null\hspace{1cm}(2) For any $\bm x, \bm x'\in \T^{2d}$, $p,q, p', q'\in \R^{d}$ and $\mu,\mu'\in\mathcal{P}(\T^d)$ we have
\begin{equation}\label{A2'}
\begin{split}
|H(\bm x,p,q,\mu)-&H(\bm x',p',q',\mu')|\\
&\leq C_H(1+|p|+|q|+|p'|+|q'|)\left(|\bm x-\bm x'|+|p-p'|+|q-q'|+{\bf d}_1(\mu,\mu')\right),
\end{split}
\end{equation}

\null\hspace{1cm}(3) For any $\mu\in\mathcal{P}(\T^d)$ the map $(x,y,p,q)\mapsto H(x,y,p,q,\mu)$ is a $C^2$ function with its second derivative being locally $\beta$-H\"older continuous uniformly in $\mu$. Furthermore, for any $x, y\in \T^d$, $p,q\in \R^d$ and $\mu\in\mathcal{P}(\T^d)$
\begin{align}
|D_xH(x,y,p,q,\mu)|&\leq C_H(1+|p|),\label{A2x}\\
|D_pH(x,y,p,q,\mu)|&\leq C_H(1+|p|).\label{A2p}
\end{align}

(A4) $\mathcal{G}: \mathcal{P}(\mathbb{T}^d)\rightarrow \R$ is ${\bf d}_1$-Lipschitz continuous, i.e. there exists a constant $C_{G}$ such that
$$|\mathcal{G}(\mu)-\mathcal{G}(\mu')|\leq C_{G} {\bf d}_1(\mu,\mu'), \text{ for every }\mu,\mu'\in \mathcal{P}(\mathbb{T}^d).$$
\vspace{-3mm}

(A5) $\sigma_1, \sigma_2,\sigma_{0,1},\sigma_{0,2}$ are $d\times d$ matrices and $\sigma_1,\sigma_2$ are invertible.\\

\noindent
We will denote the set of assumption (A1)-(A5) as assumption \textbf{(A)}.
\begin{remark}
(i) Assumptions (A4), (A5) and (A3)(2) guarantee the existence of a viscosity solution to \eqref{HJ1} (see section \ref{viscosity}).\\
    (ii) There are many examples of Hamiltonians $H$ that satisfy (A3). Three of them are 
    \begin{itemize}
        \item $H\equiv 0$.
        \item $H(x,y,p,q,\mu)=p\cdot c(x,y,\mu)+q\cdot f(x,y,\mu)$ with $L\equiv 0$ and $f,d$ independent of $\alpha$.
        \item We write $\alpha=(\alpha_1,\alpha_2)\in \R^d\times\R^d$. Then we consider $L(x,y,\alpha,\mu)=\frac{|\alpha|^2}{2}$, $c(x,y,\alpha,\mu)=\alpha_1$, $f(x,y,\alpha,\mu)=\alpha_2$, $H$ satisfies \eqref{Hamiltonian} and $A$ is closed ball centered at the origin.
    \end{itemize}
\noindent
(iii) Due to the compactness of $\mathcal{P}(\T^d)$, assumption (A4) also implies that $\mathcal{G}$ is bounded.

   
\end{remark}

\section{On viscosity solutions} \label{viscosity}
\noindent
To provide more generality to our statements, for this section only we change the notation, but we keep the same symbols for the Hamiltonian, the running cost $\mathcal{F}$ and the final cost $\mathcal{G}$. Let $n\in \mathbb{N}^*$, $H:\mathbb{T}^n\times \R^n\times\mathcal{P}(\T^n)\rightarrow \R$, $\mathcal{G}:\mathcal{P}(\mathbb{T}^n)\rightarrow \R$, an $n\times n$ invertible matrix $\sigma$ and an $n\times n$ matrix $\sigma_0$. We consider the Hamilton-Jacobi equation posed in $[0,T]\times\mathcal{P}(\mathbb{T}^n)$
\be \label{HJ}
\begin{cases}
    -\partial_t U(t,m) +\int_{\mathbb{T}^n}H\left(x,D_mU(t,m,x),m\right)m(dx)\\
    \hspace{2cm}-\int_{\mathbb{T}^n}\text{tr}\left( (\sigma\sigma^{\top}+\sigma_0\sigma_0^\top)D_xD_mU(t,m,x)\right)m(dx)\\
    \hspace{3cm}-\int_{\mathbb{T}^n}\int_{\T^n}\text{tr}\left( \sigma_0\sigma_0^\top D^2_{mm}U(t,m,x,y)\right)\;m(dx)m(dy)= 0,\\
    U(T,m)=\mathcal{G}(m).
\end{cases}
\ee
In this section we introduce the notion of viscosity solutions of HJ equations that we will be using as well as their properties. Before writing the definition we note the observation from \cite{bayraktar2023comparison} that if $U$ is smooth and we define $\widetilde{U}:[0,T]\times \T^n\times\mathcal{P}(\T^n)\rightarrow \R$ via the formula
\be\label{changeofvar}
\widetilde{U}(t,z,m)=U(t,(\text{Id}+z)_*m),
\ee
then we have
\begin{equation}\label{finite}
\begin{split}
\text{tr}\left( \sigma_0\sigma_0^\top D^2_{zz}\widetilde{U}(t,z,m)\right)=\int_{\T^n}\text{tr}&\left(\sigma_0\sigma_0^\top D_xD_mU(t,m^z,x)\right) m^z(dx)\\
&+\int_{\T^n}\int_{\T^n}\text{tr}\left(\sigma_0\sigma_0^{\top} D^2_{mm}U(t,m^z,x,y)\right)m^z(dx)m^z(dy),
\end{split}
\end{equation}
where 
\be \label{newm}
m^z=(\text{Id}+z)_*m.
\ee
Using this observation and \eqref{HJ}, one can check that $\widetilde{U}$ satisfies
\be \label{HJz}
\begin{cases}
    -\partial_t \widetilde{U} +\int_{\mathbb{T}^n}\widetilde{H}\left(x,z,D_m\widetilde{U}(t,z,m,x),m\right)m(dx)\\
    \hspace{2cm}-\int_{\mathbb{T}^n}\text{tr}\left( \sigma\sigma^{\top}D_xD_m\widetilde{U}(t,z,m,x)\right)m(dx)\\
    \hspace{4cm}-\text{tr}\left( \sigma_0\sigma_0^\top D^2_{zz}\widetilde{U}(t,z,m)\right)= 0,\\
    \widetilde{U}(T,z,m)=\widetilde{\mathcal{G}}(z,m),
\end{cases}
\ee
for $(t,z,m)\in [0,T]\times \T^n\times \mathcal{P}(\T^n)$, where
\be\label{newH,F,G}
\widetilde{H}(x,z,p,m):=H(x+z,p,m^z),\;\;\widetilde{\mathcal{G}}(z,m):=\mathcal{G}(m^z),\;\; (x,z,p,m)\in \T^n\times\T^n\times \R^n\times\mathcal{P}(\T^n).
\ee
We define viscosity solutions for the original equation \eqref{HJ} in terms of viscosity solutions of the transformed equation \eqref{HJz}. The latter is much easier to deal with, because the terms involving second-order derivatives in measure have been exchanged for a finite dimensional operator \eqref{finite}. Viscosity solutions for \eqref{HJz} are defined in the sense of test functions, in the spirit of Crandall-Lions.

\begin{definition}\label{test functions}
We say that $\varphi: [0,T]\times \T^n\times  \mathcal{P}(\mathbb{T}^n)\rightarrow \R$ is a test function and we write $\varphi\in C^{1,2,2}$, if $\partial_t\varphi(t,z.m), D_m\varphi(t,z,m,x), D_xD_m\varphi(t,z,m,x), D^2_{mm}\varphi(t,z,m,x,y)$ exist and are continuous as well as $D_z\varphi(t,z,m),D^2_{zz}\varphi(t,z,m)$ exist and are continuous.


\end{definition}

\begin{remark}
We keep the above definition for test functions $\varphi:[0,T]\times \T^{n}\times\T^n\times \mathcal{P}(\mathbb{T}^n\times \mathbb{T}^n )\rightarrow \R$, where $\mathbb{T}^n$ is substituted with $\mathbb{T}^n\times \mathbb{T}^n$ and $\mathcal{P}(\mathbb{T}^n)$ with $\mathcal{P}(\mathbb{T}^n\times \mathbb{T}^n)$. We denote the set of such test functions by $C^{1,2,2}$ again.
\end{remark}

\noindent
Having this in mind, we introduce the notion of viscosity solutions. We will be using the same definition of viscosity solutions (with the obvious modifications) for \eqref{HJ1} with $2d$ in place of $n$.

\begin{definition}\label{defz}
An upper (resp. lower) semi-continuous function $\widetilde{U}:[0,T]\times\T^n\times  \mathcal{P}(\mathbb{T}^n)\rightarrow \R$ is a viscosity subsolution (resp. supersolution) to equation \eqref{HJz} if:
\begin{itemize}
    \item $\widetilde{U}(T,z,m)\leq (\text{resp. }\geq)\widetilde{\mathcal{G}}(z,m),\text{ for every }(z,m)\in\T^n\times \mathcal{P}(\mathbb{T}^n);$
    \item for every $(t,z,m)\in [0,T)$ and any test function $\varphi\in C^{1,2,2}$ such that $\widetilde{U}-\varphi$ has a maximum (resp. minimum) at $(t,z,m)$, the first equation in \eqref{HJz} is satisfied with the inequality $\leq$ (resp. $\geq$) instead of the equality and $\varphi$ in place of $\widetilde{U}$.
    \end{itemize}
Finally, $\widetilde{U}$ is a viscosity solution of \eqref{HJz} if it is both a viscosity subsolution and a viscosity supersolution.
\end{definition}

\begin{remark}\label{strict}
   (i) Using a negative Sobolev space argument as in \cite[Lemma 2.3]{daudin2023well}, we can infer that just like in finite dimensions, in order to check that a function $U$ is a viscosity sub/super-solution to \eqref{HJz}, it suffices to restrict our attention to strict local maxima/minima. We may also assume that the strict local maximum/minimum is $0$ by shifting the test function if necessary.\\
   \noindent
(ii) Due to the analysis done at the beginning of the section, if $\sigma_{0,1},\sigma_{0,2}$ are both the zero matrix, then there is no need to introduce \eqref{HJz} into the definition of viscosity solutions. Thus, Definition \ref{vdef} will be in the sense of test functions as in the classical definition of viscosity solutions of Crandall-Lions in finite dimensions.
\end{remark}

\begin{definition}\label{vdef}
An upper (resp. lower) semicontinuous function $U:[0,T]\times \mathcal{P}(\T^n)\rightarrow \R$ is called a viscosity subsolution (resp. supersolution) of \eqref{HJ} if $\widetilde{U}:[0,T]\times \T^n\times\mathcal{P}(\T^n)\rightarrow \R$ defined as in \eqref{changeofvar} is a viscosity subsolution (resp. supersolution) of \eqref{HJz}. $U$ is a viscosity solution of \eqref{HJ} if it is a viscosity subsolution and a viscosity supersolution.
\end{definition}

\noindent
We note that from \cite[Theorem 2.6, Theorem 2.7]{daudin2023well},
we have the following result for viscosity solutions of \eqref{HJ}. We assume that $H$ satisfies the following assumption (A0), which is comparable to (A3)(2) from section \ref{assu}.
\vspace{2mm}

\noindent
\textbf{Assumption} (A0)  \hspace{0.2cm} There exists a constant $C_H'$ such that or every $(x,p,m), (x',p',m')\in \T^n\times \R^n\times \mathcal{P}(\T^n)$
\be\label{A2''}
\left| H(x,p,m)-H(x',p',m')\right|\leq C_H'(1+|p|+|p'|)\left( |x-x'|+|p-p'|+{\bf d}_1(m,m')\right).
\ee

\begin{theorem} \label{joe}
Suppose that Assumption (A0) holds and that $\mathcal{G}:\mathcal{P}(\T^n) \rightarrow \R$ is Lipschitz continuous. Then, there exists a viscosity solution $U$ of \eqref{HJ}.
Furthermore, if $U_1$ is a viscosity subsolution of \eqref{HJ} and $U_2$ a viscosity supersolution of \eqref{HJ}, then $U_1\leq U_2$ in $[0,T]\times \mathcal{P} (\T^n)$. In particular, there is at most one viscosity solution of \eqref{HJ}.
\end{theorem}
\vspace{2mm}

\noindent
A Corollary of Theorem \ref{joe} is the following.

\begin{corollary}\label{unbound}
    Assume (A0) and $\mathcal{G}$ is Lipschitz and let $U$ be the unique viscosity solution of \eqref{HJ}. Then, there exists a constant $C>0$ depending only on $\max_\mu|\mathcal{G}(\mu)|$ and $||H(\cdot,0,\cdot)||_{L^{\infty}}$ such that $|U(t,m)|\leq C$ for every $t\in [0,T]$ and $m\in \mathcal{P}(\mathbb{T}^n)$.
\end{corollary}
\begin{proof}
We will prove that there exists a constant $C>0$ depending only on $\mathcal{G}$ and $||H(\cdot,0)||_{L^{\infty}}$  such that $-C\leq \widetilde{U}(t,z,m)\leq C$ for all $(t,z,m)\in [0,T]\times \T^n\times \mathcal{P}(\T^n)$, where $\widetilde{U}$ is as in \eqref{changeofvar}. We start with the upper bound.\\
Let $C_1,C_2$ be two numbers that will be chosen later and assume that $\widetilde{U}(t,z,m)+C_1t+C_2$ attains its maximum at $(t_0,z_0,m_0)\in [0,T]\times\T^n\times \mathcal{P}(\mathbb{T}^n)$. The choice of $C_2$ is made such that the maximum of $\widetilde{U}(t,z,m)+C_1t+C_2$ is $0$. There are two cases: either $t_0\in [0,T)$ or $t_0=T$.\\
In the first case, by the definition of viscosity solutions for \eqref{HJz} with test function $-C_1t-C_2$ we get
    $$C_1+\int_{\mathbb{T}^n} \widetilde{H}(x,z_0,0,m_0)m_0(dx)\leq 0,$$
    thus $C_1\leq ||H(\cdot,0,\cdot)||_{L^{\infty}}$. This is a contradiction if $C_1$ is chosen to be large enough, therefore we may assume that $t_0=T$. We have
    $$\widetilde{U}(t,z,m)+C_1t+C_2\leq \widetilde{\mathcal{G}}(z_0,m_0)+C_1T+C_2,\;\;\text{for all } (t,z,m)\in [0,T]\times\T^n\times \mathcal{P}(\mathbb{T}^n).$$
    So, $\widetilde{U}(t,z,m)\leq \max_{m\in\mathcal{P}(\mathbb{T}^n)}\mathcal{G}(m)+|C_1|T$. The upper bound follows. To prove the lower bound we use a similar argument for the minimum of $U(t,m)-C_1t-C_2$. 
\end{proof}

\section{Proof of the main result}

\noindent
We start by presenting a formal calculation that motivates our argument. 

\subsection{Formal calculation}
Let's assume that $U^{\e}(t,m)=U(t,\mu)+\e V(t,m)$, where $\mu$ is the $x-$marginal of $m$ ($\mu=\pi^1_*m$) and $U, V$ are two regular functions (to be chosen later). The idea is that $U^{\e}$ will be an $\e$-perturbation of its limit $U$. Substituting into \eqref{HJ1} and using Proposition \ref{Wass reg} yields
\begin{align*}
-\partial_t U(t,\mu)&-\e\partial_t V(t,m)\\
&+\int_{\mathbb{T}^{2d}}H\left(x,y,D_{\mu}U(t,\mu,x)+\e D_x\frac{\delta V}{\delta m}(t,m,x,y), D_y \frac{\delta V}{\delta m}(t,m,x,y),\mu\right)m(dx,dy)\\
&-\int_{\mathbb{T}^d}\text{tr}\left(\sigma_1\sigma_1^{\top}D_xD_{\mu}U(t,\mu,x)\right)\mu(dx)-\e \int_{\mathbb{T}^{2d}}\text{tr}\left(\sigma_2\sigma_2^{\top}D^2_{xx}\frac{\delta V}{\delta m}(t,m,x,y)\right)m(dx,dx)\\
&-\int_{\mathbb{T}^{2d}}\text{tr}\left(\sigma_2\sigma_2^{\top}D^2_{yy}\frac{\delta V}{\delta m}(t,m,x,y)\right)m(dx,dx)-\int_{\T^{2d}}\text{tr}\left( \sigma_{0,1}\sigma_{0,1}^\top D_xD_{\mu}U(t,\mu,x)\right)\mu(dx)\\
&-\e \int_{\T^{2d}}\text{tr}\left(\sigma_0^{\e}\sigma_0^{\e\top} D_{(x,y)}D_mV(t,m,x,y)\right)m(dx,dy)\\
&-\int_{\T^{2d}}\int_{\T^{2d}}\text{tr}\left( \sigma_{0,1}\sigma_{0,1}^{\top}D^2_{\mu\mu}U(t,\mu,x,x') \right)\mu(dx)\mu(dx')\\
&-\e\int_{\T^{2d}}\int_{\T^{2d}}\text{tr}\left(  \sigma_0^{\e}\sigma_0^{\e\top}D^2_{mm}V(t,m,x,y,x',y')\right)m(dx,dy)m(dx',dy')=0.
\end{align*}
By the regularity of $H$, the third term is equal to 
$$\int_{\mathbb{T}^{2d}}H\left(x,y,D_{\mu}U(t,\mu,x), D_y \frac{\delta V}{\delta m}(t,m,x,y),\mu\right)m(dx,dy)+\mathcal{O}(\e),$$
therefore collecting the $\mathcal{O}(\e)$ terms, the equation becomes
\begin{align}
-\partial_t U(t,\mu)&+\int_{\mathbb{T}^{2d}}H\left(x,y,D_{\mu}U(t,\mu,x), D_y \frac{\delta V}{\delta m}(t,m,x,y),\mu\right)m(dx,dy)\nonumber\\
&-\int_{\mathbb{T}^d}\text{tr}\left(\sigma_1\sigma_1^{\top}D_xD_{\mu}U(t,\mu,x)\right)\mu(dx) -\int_{\mathbb{T}^{2d}}\text{tr}\left(\sigma_2\sigma_2^{\top}D^2_{yy}\frac{\delta V}{\delta m}(t,m,x,y)\right)m(dx,dx)\nonumber\\
&-\int_{\T^{2d}}\int_{\T^{2d}}\text{tr}\left( \sigma_{0,1}\sigma_{0,1}^{\top}D^2_{\mu\mu}U(t,\mu,x,x,') \right)\mu(dx)\mu(dx')\nonumber\\
&-\int_{\T^{2d}}\text{tr}\left( \sigma_{0,1}\sigma_{0,1}^\top D_xD_{\mu}U(t,\mu,x)\right)\mu(dx)\nonumber\\
&-\int_{\T^{2d}}\text{tr}\left(\sigma_{0,2}\sigma_{0,2}^{\top}D^2_{yy}\frac{\delta V}{\delta m}(t,m,x,y)\right)m(dx,dy)\nonumber\\
& -\int_{\T^{2d}}\int_{\T^{2d}}\text{tr}\left(\sigma_{0,2}\sigma_{0,2}^{\top}D^2_{yy'}\frac{\delta^2 V}{\delta m^2}(t,m,x,y,x',y') \right)m(dx,dy)m(dx',dy')=\mathcal{O}(\e).\label{formal}
\end{align}
As $U$ is expected to be the limit of $U^{\e}$ as $\e\rightarrow 0$, to find the equation satisfied by it we need to eliminate $V$ from \eqref{formal}. To do that, we would like to have
\begin{align*}
\int_{\T^{2d}}&H\left(x,y,D_{\mu}U(t,\mu,x), D_y \frac{\delta V}{\delta m}(t,m,x,y),\mu\right)m(dx,dy)\\
&-\int_{\T^{2d}}-\text{tr}\left((\sigma_2\sigma_2^{\top}+\sigma_{0,2}\sigma_{0,2}^{\top})D^2_{yy}\frac{\delta V}{\delta m}(t,m,x,y)\right) m(dx,dy)\\
&\;\;\;\;\;\;-\int_{\T^{2d}}\int_{\T^{2d}}\text{tr}\left(\sigma_{0,2}\sigma_{0,2}^{\top}D^2_{yy'}\frac{\delta^2 V}{\delta m^2}(t,m,x,y,x',y')\right)m(dx,dy)m(dx',dy')=B(\mu),
\end{align*}
for some function $B:\mathcal{P}(\T^d)\rightarrow \R$ depending on $U$. This leads us to considering the problem
\be \label{cell}
\begin{split}
H(x,y,D_{\mu}U(t,\mu,x)&, D_y \phi(t,\mu,x,y),\mu)\\
&-\text{tr}\left( (\sigma_2\sigma_2^{\top}+\sigma_{0,2}\sigma_{0,2}^{\top})D^2_{yy}\phi(t,\mu,x,y)\right)=\overline{H}(x,D_{\mu}U(t,\mu,x),\mu),
\end{split}
\ee
where on the right hand side we have the unique constant such that this problem has a unique up to a constant solution $\phi$. Since we want $D_y\phi = D_y\frac{\delta V}{\delta m}$, we let 
$$V(t,m)=\int_{\mathbb{T}^{2d}}\phi(t,\mu,x,y)m(dx,dy).$$
Then, again by Proposition \ref{Wass reg},
$\frac{\delta V}{\delta m}(t,m,x,y)= \phi(t,\mu,x,y)+\int_{\mathbb{T}^{2d}}\frac{\delta \phi}{\delta \mu}(t,\mu,x',y',x)m(dx',dy')$,
so the Wasserstein partial derivative $D_y\frac{\delta V}{\delta m}(t,m,x,y)=D_y\phi(t,\mu,x,y)$, which implies
\begin{align*}
\int_{\mathbb{T}^{2d}}\overline{H}(x,D_{\mu}U(t,\mu,x),\mu&)\mu(dx)= \int_{\mathbb{T}^{2d}}H\left(x,y,D_{\mu}U(t,\mu,x), D_y \frac{\delta V}{\delta m}(t,m,x,y),\mu\right)m(dx,dy)\\
&-\int_{\T^{2d}}\text{tr}\left((\sigma_2\sigma_2^{\top}+\sigma_{0,2}\sigma_{0,2}^{\top})D^2_{yy}\frac{\delta V}{\delta m}(t,m,x,y)\right)m(dx,dy)\\
&-\int_{\T^{2d}}\int_{\T^{2d}}\text{tr}\left(\sigma_{0,2}\sigma_{0,2}^{\top}D^2_{yy'}\frac{\delta^2 V}{\delta m^2}(t,m,x,y,x',y')\right)m(dx,dy)m(dx',dy'),
\end{align*}
due to \eqref{cell} and the fact that $\frac{\delta^2V}{\delta m^2}$ does not have a term depending on both $y$ and $y'$. Therefore, for that choice of $V$ \eqref{formal}, gives
\begin{align}
-\partial_t U(t,\mu)&+\int_{\mathbb{T}^{2d}}\overline{H}(x,D_{\mu}U(t,\mu,x),\mu)\mu(dx)-\int_{\mathbb{T}^d}\text{tr}\left((\sigma_1\sigma_1^{\top}+\sigma_{0,1}\sigma_{0,1}^{\top})D_xD_{\mu}U(t,\mu,x)\right)\mu(dx)\nonumber\\
& -\int_{\T^d}\int_{\T^d} \text{tr}\left( \sigma_{0,1}\sigma_{0,1}^\top D^2_{\mu\mu}U(t,\mu,x,x') \right) \mu(dx)\mu(dx')=\mathcal{O}(\e).\nonumber
\end{align}
We send $\e\rightarrow 0$ to derive that $U$ satisfies equation \eqref{HJ2}.

\subsection{Proof of Theorem \ref{main}}
As the formal calculation suggests, for $(x,p,\mu)\in \T^d\times\R^d\times\mathcal{P}(\T^d)$, to construct $\overline{H}(x,p)$ we need to study the following problem, which is posed in $\mathbb{T}^d$:
$$
H(x,y,p,D_yv,\mu)-\text{tr}\left((\sigma_2\sigma_2^{\top}+\sigma_{0,2}\sigma_{0,2}^{\top})D^2_{yy}v\right)=l.$$
Such problems have been studied in the context of homogenization theory, for example in \cite{evans1992periodic}, and it is known that there exists a unique constant $l$ such that \eqref{cell} is well posed (up to an additive constant). We provide an argument on the existence of such a constant and we show some properties that will be useful in the proof of the main Theorem \ref{main}. In particular, the following Proposition holds.

\begin{proposition}\label{lemma1}
    Assume that $H$ satisfies (A3). Then, for any $\delta>0$ and $x\in\T^d,\; p\in \R^d$, $\mu\in\mathcal{P}(\T^d)$, there exists a unique classical solution $v^{\delta}:=v^{\delta}(x,p,y,\mu)$ of 
    \be \label{delta}\delta v^{\delta}-\text{tr}\left((\sigma_2\sigma_2^{\top}+\sigma_{0,2}\sigma_{0,2}^{\top})D^2_{yy} v^{\delta}\right)+H(x,y,p,D_yv^{\delta},\mu)=0,\;\; y\in \mathbb{T}^d.\ee
    The function $v^{\delta}(x,p,y,\mu)\in C^2_{x,p,y}$, that is the derivatives $D_{x}v^{\delta}, D^2_{xx}v^{\delta}, D_pv^{\delta}, D_{pp}^2v^{\delta}, D^2_{xp}v^{\delta}$ exist and are continuous, and for any $(x,p,\mu)\in \T^d\times \R^d\times\mathcal{P}(\T^d)$ there exists a number $\overline{H}(x,p,\mu)$ such that 
    \be\label{effectiveapp}
\lim_{\delta\rightarrow 0}\left\| \delta v^{\delta}(x,p,\cdot,\mu)+\overline{H}(x,p,\mu)\right\|_{L^{\infty}}=0.
    \ee
Furthermore, the function $\overline{H}:\T^d\times \R^d\times\mathcal{P}(\T^d)\rightarrow \R$ satisfies for some constant $C>0$
\be\label{regofeff}
\left|\overline{H}(x,p,\mu)-\overline{H}(x',p',\mu')\right|\leq C(1+|p|+|p'|)\left(|x-x'|+|p-p'|+{\bf d}_1(\mu,\mu')\right),
\ee
$\text{for all } x,x'\in \T^d,\; p,p'\in \R^d\;\text{and } \mu,\mu'\in\mathcal{P}(\T^d).$

\end{proposition}

\vspace{3mm}

\begin{proof}
Throughout the proof, we use the notation $a=\sigma_2\sigma_2^{\top}+\sigma_{0,2}\sigma_{0,2}^{\top}$, which is a strictly positive matrix due to assumption (A5). The existence and the uniqueness of a classical solution $v^{\delta}(x,p,\cdot,\mu)$, where $x\in\T^d,\; p\in \R^d$ and $\mu\in\mathcal{P}(\T^d)$ are fixed, is a standard result and can be found in \cite{gilbarg1977elliptic}. It suffices to show the other statements. We will divide the proof in 4 steps.
\vspace{1mm}

\noindent
\textit{Step 1.} (estimates on $v^{\delta}$ and existence of $\overline{H}(x,p,\mu)$)\\
Let $\lambda(p)=\| H(\cdot,\cdot, p,0,\cdot)\|_{L^{\infty}}$. We will show that there exists $\beta\in (0,1)$ and a constant $C>0$ independent of $x,p,\mu$ such that
\be\label{estimate}
\left\|v^{\delta}(x,p,\cdot,\mu)-v^{\delta}(x,p,0,\mu)\right\|_{C^{2,\beta}}\leq C(1+\lambda(p)+|p|).
\ee
Our argument follows that in \cite{arisawa1998ergodic}. Considering the optimality conditions for a maximum or a minimum of $v^{\delta}$ in \eqref{delta}, we deduce 
\be\label{optmax}
\left\|\delta v^{\delta}(x,p,\cdot,\mu)\right\|_{L^{\infty}}\leq \lambda(p).
\ee
We claim the following uniform (in $\delta$) bound
\be\label{unifbound}
\left\|v^{\delta}(x,p,\cdot,\mu)-v^{\delta}(x,p,0,\mu)\right\|_{L^{\infty}}\leq C(1+\lambda(p)+|p|).
\ee
If we assume that \eqref{unifbound} does not hold, then there exists a sequence $(\delta_k,x_k,p_k,\mu_k)_{k\in\mathbb{N}}$ with $\delta_k\rightarrow 0$ such that the solution $v_k(\cdot)=v^{\delta_k}(x_k,p_k,\cdot,\mu_k)$ satisfies
$$\left\| v_k-v_k(0)\right\|_{L^{\infty}}\geq k(1+\lambda(p_k)+|p_k|).$$
We set $b_k=\left\| v_k-v_k(0)\right\|_{L^{\infty}}^{-1}$ and $\tilde{v}_k=b_k(v_k-v_k(0))$. Then, $\tilde{v}_k(0)=0$, $\|\tilde{v}_k\|_{L^{\infty}}=1$ and $\tilde{v}_k$ satisfies
\begin{align*}
\delta_k \tilde{v}_k+b_k\delta_k v_k(0)-&\text{tr}\left(a D^2_{yy} \tilde{v}_k\right)\\
&+\sup_{\alpha\in A}\left\{ -b_kp_k\cdot c(x_k,y,\alpha,\mu_k)-D_y\tilde{v}_k\cdot f(x_k,y,\alpha,\mu_k)-b_kL(x_k,y,\alpha,\mu_k)\right\}=0. 
\end{align*}
Since $L,c$ are Lipschitz with respect to $y$, we have for $\beta\in (0,1)$ and due to \eqref{optmax}
$$|b_k\delta_kv_k(0)|+\left\| b_kp_k\cdot c(x_k,\cdot,\alpha,\mu_k)\right\|_{C^{0,\beta}}+\left\|b_k L(x_k,\cdot,\alpha,\mu_k)\right\|_{C^{0,\beta}}\leq Cb_k(1+|p_k|+\lambda(p_k))\leq \frac{C}{k}.$$
Now, by the regularity theory of quasilinear elliptic equations \cite{gilbarg1977elliptic, safonov1989classical}, we deduce that $\tilde{v}_k$ is uniformly bounded in $C^{2,\beta}$ for some $\beta\in (0,1)$. Fix $\alpha_0\in A$. The families of functions $\{\tilde{v}_k(\cdot)\}$, $\{ c(x_k,\cdot,\alpha_0,\mu_k)\}$, $\{ f(x_k,\cdot,\alpha_0,\mu_k)\}$ are uniformly Lipschitz and equibounded, hence the Arzela-Ascoli theorem yields the existence of a limit (along subsequences) to $\tilde{v}, \tilde{c}, \tilde{f}$, respectively, where $\tilde{v}\in C^{2,\beta}$. It is clear that $\tilde{v}(0)=0$, $\|\tilde{v}\|_{L^{\infty}}=1$ and $\tilde{v}$ satisfies
\be \label{subsolmax}
-\text{tr}\left(a D^2_{yy}\tilde{v}\right)-D_y\tilde{v}\cdot \tilde{f}(y)\leq 0.
\ee
Notice that $\tilde{v}$ is a periodic function, thus by the strong maximum principle in \eqref{subsolmax}, it must be constant. This contradicts the previous equalities.
\vspace{1mm}

\noindent
Now, we may use the estimates for quasilinear elliptic equations for the equation satisfied by $v^{\delta}$ and assumptions (A1), (A2) to get
\begin{align*}
\left\| v^{\delta}-v^{\delta}(0)\right\|_{C^{2,\beta}}&\leq C\left( \left\| v^{\delta}-v^{\delta}(0)\right\|_{L^{\infty}}+|\delta v^{\delta}(0)|+\sup_{\alpha\in A}\|L(x,\cdot,\alpha,\mu)\|_{C^{0,\beta}}+\sup_{\alpha\in A}\|p\cdot c(x,\cdot,\alpha,\mu)\|_{C^{0,\beta}}\right)\\
&\stackrel{\eqref{unifbound},\eqref{optmax}}{\leq} C(1+|p|+\lambda(p)),\end{align*}
which is \eqref{estimate}.
\vspace{1mm}

\noindent
With \eqref{estimate} and \eqref{optmax} in mind, we can use Arzela-Ascoli once again to the uniformly bounded and equicontinuous family $\{v^{\delta}(x,p,\cdot,\mu)-v^{\delta}(x,p,0,\mu)\}$ to extract a subsequence such that $v^{\delta}(x,p,y,\mu)-v^{\delta}(x,p,0,\mu)\rightarrow v(x,p,y,\mu)$ in $C^2$ as $\delta\rightarrow 0$ and $\delta v^{\delta}(x,p,0,\mu)\rightarrow -\overline{H}(x,p,\mu)$ (bounded sequence has convergent subsequence). Moreover, $v$ is a classical solution of 
\be\label{limitpde}
-\text{tr}\left( aD^2_{yy}v\right)+H(x,y,p,D_yv,\mu)+\overline{H}(x,p,\mu)=0.
\ee
However, it is straightforward to show that there is a unique $\overline{H}(x,p,\mu)$ such that \eqref{limitpde} has a unique solution up to a constant, hence $\overline{H}(x,p,\mu),v$ are unique subsequential limits of $\delta v^{\delta}(0),\; v^{\delta}-v^{\delta}(0)$, respectively. Finally, \eqref{effectiveapp} follows by the triangle inequality and the aforementioned convergences.

\vspace{2mm}
\noindent
\textit{Step 2.} (first derivatives of $v^{\delta}$)\\
Throughout the proof we still consider $(x,p,\mu)\in \T^d\times \R^d\times\mathcal{P}(\T^d)$ to be fixed. 
\vspace{1mm}

\noindent
We will start with the derivatives with respect to $x$. We will show that the partial derivatives $\partial_{x_i}v^{\delta}(x,p,y,\mu)$ exist for every $i=1,...,d$. Let $i\in \{1,...,d\}$ and $h$ be a vector which is parallel to $(0,... ,0,1,0,..,0)\in \R^d$, where the $1$ is in the $i$-th position. The difference $V_i^{\delta}(y):=v^{\delta}(x+h,p,y,\mu)-v^{\delta}(x,p,y,\mu)$ satisfies the equation
\be \label{x-diff}
\begin{split}
\delta V_i^{\delta} -\text{tr}\left(a D^2_{yy}V_i^{\delta}\right)+\int_0^1H_{q}&\left(sh+x,y,p,sD_yV_i^{\delta}+D_yv^{\delta}(x,p,y),\mu\right)ds\cdot D_yV_i^{\delta}\\
&=-|h|\int_0^1 H_{x_i}\left(sh+x,y,p,sD_yV_i^{\delta}+D_yv^{\delta}(x,p,y),\mu\right)ds.
\end{split}
\ee
Let $y_1,y_2\in \T^d$ be a position of maximum and a position of minimum of $V^{\delta}$, respectively. By combining the optimality conditions with \eqref{x-diff} we obtain
\begin{align*}
&\frac{|h|}{\delta}\int_0^1 H_{x_i}\left( sh+x,y_2,p,D_yv^{\delta}(x,p,y_2),\mu\right)ds\leq \min_y V_i^{\delta}\quad\text{ and }\\
&\max_y V_i^{\delta}\leq \frac{|h|}{\delta}\int_0^1 H_{x_i}\left( sh+x,y_1,p,D_yv^{\delta}(x,p,y_1),\mu\right)ds,
\end{align*}
hence, due to \eqref{estimate} and assumption (A3)(3) (equation \eqref{A2x}) on $H$,
\be \label{infinityestimate}
\| V_i^{\delta}(\cdot) \|_{L^{\infty}}\leq \frac{C_H}{\delta}(1+|p|)|h|=C'|h|,
\ee
for some constant $C'>0$. We also observe that, due to \eqref{estimate} and the assumptions on $H$, the two integrals in \eqref{x-diff} are Lipschitz with respect to the $y$ variable, therefore by the Schauder estimates for elliptic equations (see \cite[Chapter 6]{gilbarg1977elliptic}) and \eqref{infinityestimate}, there exists a constant $C>0$ such that $\|V_i^{\delta}\|_{C^{2,\beta}}\leq C|h|$. By Arzela-Ascoli, this implies that $V_i^{\delta}/|h|$ has a convergent subsequence (in $C^2$) as $|h|\rightarrow 0$ to a function $w_i^{\delta}(x,p,y,\mu)$. By sending $|h|\rightarrow 0$ (along the subsequence) in \eqref{x-diff}, we get that the limit $w_i^{\delta}(x,p,\cdot,\mu)\in C^2$ satisfies the equation
\be\label{x-dereq}
\delta w_i^{\delta} -\text{tr}\left(aD^2_{yy}w_i^{\delta}\right)+H_{q}\left(x,y,p,D_yv^{\delta}(x,p,y),\mu\right)\cdot D_yw_i^{\delta}=-H_{x_i}\left(x,y,p,D_yv^{\delta}(x,p,y),\mu\right),
\ee
in the classical sense. Since \eqref{x-dereq} has a unique solution, we deduce $V_i^{\delta}(y)/|h|\xrightarrow{|h|\rightarrow 0} w_i^{\delta}(x,p,y,\mu)$ in $C^2$. Thus, $\partial_{x_i}v^{\delta}(x,p,y,\mu)=w_i^{\delta}(x,p,y,\mu)$. We also note that $\partial_{y_j}\partial_{x_i}v^{\delta}(x,p,y,\mu)=\partial_{y_j}w_i^{\delta}(x,p,y,\mu)$ for $j=1,...,d$. Furthermore, by the $C^2$ convergence we get $\partial_{x_i}\partial_{y_j}v^{\delta}=\partial_{y_j}w_i^{\delta}(x,p,y,\mu)$ for $j=1,...,d$. By repeating the same argument in all directions, we conclude that $D_xv^{\delta}(x,p,y,\mu)$ and $D^2_{yx}v^{\delta}(x,p,y,\mu)=D^2_{xy}v^{\delta}(x,p,y,\mu)$ exist and are continuous in $x,y$.
\vspace{2mm}

\noindent
For the derivative with respect to $p$ we use a similar strategy. Let $i\in \{1,...,d\}$ and $h$ be a vector which is parallel to $(0,...,0,1,0,...,0)\in \R^d$, where $1$ is in the $i$-th position, with $|h|<1$. The difference $\tilde{V}_i^{\delta}(y):=v^{\delta}(x,p+h,y,\mu)-v^{\delta}(x,p,y,\mu)$ satisfies
\be \label{p-diff}
\begin{split}
    \delta\tilde{V}_i^{\delta}-\text{tr}\left(aD^2_{yy}\tilde{V}_i^{\delta}\right)+\int_0^1H_{q}&\left(x,y,sh+p,sD_y\tilde{V}_i^{\delta}+D_yv^{\delta}(x,p,y),\mu\right)ds\cdot D_y\tilde{V}_i^{\delta}\\
    &=-|h| \int_0^1H_{p_i}\left(x,y,sh+p,sD_y\tilde{V}_i^{\delta}+D_yv^{\delta}(x,p,y),\mu\right)ds.
\end{split}
\ee
By arguing as in the $x$-derivative case, by assumption (A3)(3) (equation \eqref{A2p}) on $H$ we have
\be\label{pinfinityestimate}
\|\tilde{V}_i^{\delta}(\cdot)\|_{L^{\infty}}\leq \frac{C_H}{\delta}(1+|p|+|h|)|h|\leq C'|h|,
\ee
for some constant $C'>0$, since $|h|<1$ and $p$ is fixed. In addition, the two integrals in \eqref{p-diff} are $\beta$-H\"older continuous with respect to $y$ (recall assumption (A3)(3)), therefore by the Schauder estimates, we also have $\|\tilde{V}_i^{\delta}\|_{C^{2,\beta}}\leq C|h|$ for some constant $C>0$. Now, as in the $x$-derivative case we can argue that $\tilde{V}_i^{\delta}/|h|$ converges in $C^2$, as $|h|\rightarrow 0$, to the unique classical solution of 
\be\label{p-dereq}
\delta \tilde{w}_i^{\delta}-\text{tr}\left(aD^2_{yy}\tilde{w}_i^{\delta}\right)+H_{q}\left(x,y,p,D_yv^{\delta}(x,p,y),\mu\right)\cdot D_y\tilde{w}_i^{\delta}=-H_{p_i}\left(x,y,p,D_yv^{\delta}(x,p,y),\mu\right).
\ee
Thus, since there is a unique function satisfying \eqref{p-dereq}, we have $\tilde{V}_i^{\delta}/|h|\xrightarrow{|h|\rightarrow 0} \tilde{w}_i^{\delta}$, which implies that $\partial_{p_i}v^{\delta}(x,p,y,\mu)=\tilde{w}_i^{\delta}(x,p,y,\mu)$ and $\partial_{y_j}\partial_{p_1}v^{\delta}(x,p,y,\mu)=\partial_{y_j}\tilde{w}_i^{\delta}(x,p,y,\mu)$ for $j=1,...,d$. Furthermore, by the convergence in $C^2$, we also have $\partial_{p_i}\partial_{y_j}v^{\delta}(x,p,y,\mu)=\partial_{y_j}\tilde{w}_i^{\delta}(x,p,y,\mu)$ for $j=1,...,d$. Therefore, we deduce that $D_pv^{\delta}, D^2_{py}v^{\delta},D^2_{yp}v^{\delta}$ exist and $D^2_{yp}v^{\delta}=D^2_{py}v^{\delta}$.
\vspace{2mm}

\noindent
Before continuing with the next step, we note that the bounds $\|V^{\delta}_i\|_{C^{2,\beta}}\leq C(1+|p|)|h|$ and $\|\tilde{V}^{\delta}_i\|_{C^{2,\beta}}\leq C(1+|p|)|h|$ for $i=1,...,d$ imply that $x\mapsto D_yv^{\delta}(x,p,y,\mu)$ is Lipschitz continuous with Lipschitz constant $C(1+|p|)$ and the map $p\mapsto D_yv^{\delta}(x,p,y,\mu)$ is continuous.
\vspace{2mm}

\noindent
\textit{Step 3.} (second derivatives of $v^{\delta}$)\\
We fix $i,j\in \{1,...,d\}$. We will first show that the map $x\mapsto D_yw_i^{\delta}(x,p,y,\mu)$ is continuous, where $w_i^{\delta}$ is the unique solution of \eqref{x-dereq}. Let $h$ be a vector parallel to $(0,...,0,1,0,...,0)\in\R^d$, where the $1$ is in the $j$-th position. The function $W^{\delta}_{ij}(y):=w^{\delta}_i(x+h,p,y,\mu)-w^{\delta}_i(x,p,y,\mu)$ satisfies 
\begin{align}
\delta W&^{\delta}_{ij}-\text{tr}\left(aD_{yy}^2W^{\delta}_{ij}\right)+H_q(x,y,p,D_yv^{\delta},\mu)\cdot D_yW^{\delta}_{ij}\nonumber\\
&=-|h|\int_0^1 H_{x_jx_i}\left(sh+x,y,p, sD_yV^{\delta}_j+D_yv^{\delta},\mu\right)ds\nonumber\\
&\hspace{2cm}-\int_0^1 \partial_{x_j}H_q\left(sh+x,y,p,sD_yV^{\delta}_j+D_yv^{\delta},\mu\right)ds\cdot D_yV^{\delta}_j\nonumber\\
&\hspace{2.5cm}-|h|\int_0^1\partial_{x_j}H_q\left(sh+x,y,p,sD_yV^{\delta}_j+D_yv^{\delta},\mu\right)ds\cdot D_yw_i^{\delta}(x+h)\nonumber\\
&\hspace{3cm}-\int_0^1 H_{qq}\left(sh+x,y,p,sD_yV^{\delta}_j+D_yv^{\delta},\mu\right)D_yV^{\delta}_jds\cdot D_yw_i^{\delta}(x+h)\label{2x-diff},
\end{align}
where $V^{\delta}_i$, for $i=1,...,d$, was defined right before \eqref{x-diff}. By Step 2, we have a uniform H\"older bound for $w_i^{\delta}$, namely $\|w_i^{\delta}(x)\|_{C^{2,\beta}}\leq C(1+|p|)$ for some $C>0$ independent of $x$. Combining this with \eqref{estimate} and $\|V^{\delta}\|_{C^{2,\beta}}\leq C|h|$ we have the following bounds for the terms on the right hand side of \eqref{2x-diff}
\begin{gather*}
  \bigg| |h|\int_0^1 H_{x_jx_i}\left(sh+x,y,p, sD_yV^{\delta}_j+D_yv^{\delta},\mu\right)ds\bigg| \leq C|h|,\\
  \bigg|\int_0^1 \partial_{x_j}H_q\left(sh+x,y,p,sD_yV^{\delta}_j+D_yv^{\delta},\mu\right)ds\cdot D_yV^{\delta}_j\bigg| \leq C|h|,\\
\bigg||h|\int_0^1\partial_{x_j}H_q\left(sh+x,y,p,sD_yV^{\delta}_j+D_yv^{\delta},\mu\right)ds\cdot D_yw_i^{\delta}(x+h)\bigg|\leq C|h|,\\
\bigg|\int_0^1 H_{qq}\left(sh+x,y,p,sD_yV^{\delta}_j+D_yv^{\delta},\mu\right)D_yV^{\delta}_jds\cdot D_yw_i^{\delta}(x+h)\bigg|\leq C|h|.
\end{gather*}
Thus, by the maximum principle, we have $\|W^{\delta}_{ij}(\cdot)\|_{L^{\infty}}\leq C|h|$. In addition, by the properties of $H$, all these terms are H\"older continuous in the $y$ variable, hence by the Schauder estimates we get $\|W^{\delta}_{ij}(\cdot)\|_{C^{2,\beta}}\leq C|h|$. Since $j$ was arbitrary, these estimates yield the continuity of $x\mapsto D_yw_i^{\delta}(x,p,y,\mu)$.
\vspace{2mm}

\noindent 
We introduce the function $\tilde{W}^{\delta}_{ij}(y)=\tilde{w}_i^{\delta}(x,p+h,y,\mu)-\tilde{w}_i^{\delta}(x,p,y,\mu)$, where $i,j,h$ are the same as above and $\tilde{w}_i^{\delta}$ is the solution of \eqref{p-dereq}. We will show that the map $p\mapsto D_y\tilde{w}^{\delta}_i(x,p,y,\mu)$ is continuous. Indeed, $\tilde{W}^{\delta}_{ij}$ satisfies
\begin{align}
\delta \tilde{W}^{\delta}_{ij}&-\text{tr}\left( aD^2_{yy}\tilde{W}^{\delta}_{ij}\right)+H_q(x,y,p,D_yv^{\delta}(p),\mu)\cdot D_y\tilde{W}^{\delta}_{ij}\nonumber\\
&=-|h|\int_0^1 H_{p_jp_i}\left(x,y,sh+p,sD_y\tilde{V}^{\delta}_j+D_yv^{\delta},\mu\right)ds\nonumber\\
&\hspace{1.5cm}-\int_0^1 \partial_{p_j}H_q\left(x,y,sh+p,sD_y\tilde{V}^{\delta}_j+D_yv^{\delta},\mu\right)ds\cdot D_y\tilde{V}_j^{\delta}\nonumber\\
&\hspace{2cm}-|h|\int_0^1\partial_{p_j}H_q\left(x,y,sh+p,sD_y\tilde{V}^{\delta}_j+D_yv^{\delta},\mu\right)ds\cdot D_y\tilde{w}^{\delta}_i(p+h)\nonumber\\
&\hspace{2.5cm} -\int_0^1 H_{qq}\left(x,y,sh+p,sD_y\tilde{V}^{\delta}_j+D_yv^{\delta},\mu\right)D_y\tilde{V}^{\delta}_jds\cdot D_y\tilde{w}^{\delta}_i(p+h),\label{2p-diff}
\end{align}
where $\tilde{V}^{\delta}_j$, for $j=1,...,d$ was defined in \eqref{p-diff}. Again by Step 2, we have the bounds $\|\tilde{w}_i^{\delta}\|_{C^{2,\beta}}\leq C$ and $\|\tilde{V}^{\delta}_i\|_{C^{2,\beta}}\leq C|h|$, where $C$ depends on $p$. Thus, we can show with a similar argument as before this time for equation \eqref{2p-diff} that the map $p\mapsto D_y\tilde{w}_i^{\delta}(x,p,y,\mu)$ is continuous. 
\vspace{1mm}

\noindent
Using the continuity of both maps $x\mapsto D_yw_i^{\delta}(x,p,y,\mu)$ and $p\mapsto D_y\tilde{w}_i^{\delta}(x,p,y,\mu)$, we can prove that $D^2_{xx}v^{\delta}, D^2_{xp}v^{\delta}, D^2_{px}v^{\delta}$ and $D^2_{xx}v^{\delta}$ exist and are continuous in $x,y,p$. Indeed, in light of these properties, the right hand side in \eqref{x-dereq} and the coefficient of $D_yw_i^{\delta}$ in \eqref{x-dereq} are $C^1$ functions with respect to the $x$ and $p$ variables, therefore we can apply Lemma \ref{parametric} for \eqref{x-dereq} to deduce that $(x,p)\mapsto w_i^{\delta}(x,p,y,\mu)=\partial_{x_i}v^{\delta}$ is also a $C^1$ function, $i=1,...,d$. Hence, $D^2_{xx}v^{\delta}$ and $D^2_{px}v^{\delta}$ exist and are continuous in $x,p$. We can similarly apply Lemma \ref{parametric} to \eqref{p-dereq} to deduce that $D^2_{pp}v^{\delta}$ and $D^2_{xp}v^{\delta}$ exist and are continuous in $x,p$.
\vspace{2mm}

\noindent
\textit{Step 4.} (regularity of $\overline{H}$, \eqref{regofeff})\\
Let $x,x'\in \T^d$, $p,p'\in \R^d$ and $\mu,\mu'\in\mathcal{P}(\T^d)$. By \eqref{infinityestimate} for $i=1,...,d$, we get for any $y\in \T^d$
\be\label{eq1001}
|\delta v^{\delta}(x',p',y)-\delta v^{\delta}(x,p',y)|\leq C|x'-x|(1+|p'|).
\ee
On the other hand, by \eqref{pinfinityestimate} for $i=1,...,d$ (note that the condition $|h|<1$ is used only on the second part of the inequality) we have
\be\label{eq1002}
|\delta v^{\delta}(x,p',y)-\delta v^{\delta}(x,p,y)|\leq C|p-p'|( |p|+|p'|),\text{ for any }y\in\T^d
\ee
Moreover, by considering the optimality conditions for the maximum and the minimum of the function $y\mapsto \delta v^{\delta}(x',p',y,\mu')-\delta v^{\delta}(x',p',y,\mu)$, as in Step 2, and by using assumption (A3)(2), we can prove that
\be\label{eq1003}
|\delta v^{\delta}(x',p',y,\mu')-\delta v^{\delta}(x',p',y,\mu)|\leq C(1+|p'|){\bf d}_1(\mu,\mu'),\text{ for any }y\in\T^d.
\ee
We add \eqref{eq1001}, \eqref{eq1002}, \eqref{eq1003} and we send $\delta\rightarrow 0$ to discover, by virtue of \eqref{effectiveapp},
\begin{align*}
|\overline{H}(x',p',\mu')-\overline{H}(x',p',\mu)|+|\overline{H}(x',p',\mu)- \overline{H}(&x,p',\mu)|+|\overline{H}(x,p',\mu)-\overline{H}(x,p,\mu)|\\
&\leq C_H(1+|p|+|p'|)\left(|x-x'|+|p-p'|+{\bf d}(\mu,\mu')\right).
\end{align*}
Estimate \eqref{regofeff} follows by applying the triangle inequality.

\end{proof}

\vspace{2mm}

\noindent
We are now ready to prove Theorem \ref{main}. The idea of the proof is to adapt the formal calculation from the previous subsection to Definitions \ref{defz} and \ref{vdef} and to adjust the perturbed test function method for Hamilton-Jacobi equations in finite dimensions (e.g., \cite{evans1989perturbed}) to this infinite dimensional setting.
\vspace{3mm}

\begin{proof}
Let $\overline{H}$ be the function we constructed in Proposition \ref{lemma1}. By Corollary \ref{unbound}, the family $U^{\e}(t,m)$ is uniformly bounded (the uniform bound depends only on $\mathcal{G}$ and $||H(\cdot,\cdot,0, 0,\cdot)||_{L^{\infty}}$), hence we can define the upper and lower semi-limits for $(t,\mu)\in [0,T]\times\mathcal{P}(\mathbb{T}^d)$:
\begin{align}
&\overline{U}(t,\mu)=\limsup_{\substack{\e\rightarrow 0 \\ t'\rightarrow t\\ \mu'\rightarrow \mu} }\sup_{\substack{m\in \mathcal{P}(\mathbb{T}^{2d}):\\
\pi^1_*m=\mu'}} U^{\e}(t',m) \label{semi}, \\
&\underline{U}(t,\mu)=\liminf_{\substack{\e\rightarrow 0 \\ t'\rightarrow t\\ \mu'\rightarrow \mu} }\inf_{\substack{m\in \mathcal{P}(\mathbb{T}^{2d}):\\
\pi^1_*m=\mu'}} U^{\e}(t',m). \label{lowersemi}
\end{align}
We will show that $\overline{U}$ is a viscosity subsolution of \eqref{HJ2} and that $\underline{U}$ is a viscosity supersolution of \eqref{HJ2}. Then, since \eqref{regofeff} holds, Theorem \ref{joe} applies for \eqref{HJ2}; therefore  $\underline{U}\geq \overline{U}$. Combining this with the obvious fact that $\underline{U}\leq \overline{U}$, we deduce $\underline{U}=\overline{U}:=U$. Since both semi-limits are equal to $U$, the uniform convergence of $U^{\e}$ to $U$ that we need to prove. We will only show that $\overline{U}$ is a subsolution, the proof for the supersolution being analogous.
\vspace{2mm}

\noindent
To prove the first part of the definition $\widetilde{\overline{U}}(T,z,\mu)\leq \widetilde{\mathcal{G}}(z,\mu)$ for any $(z,\mu)\in \T^d\times\mathcal{P}(\T^d)$, we rely on \cite[Theorem 4.2]{cosso2023smooth}\footnote{The result in this paper is stated for functions defined on $\mathcal{P}_2(\R^d)$, but the proof can be easily adapted for functions defined on $\mathcal{P}(\T^d)$.}
to get a sequence $(\mathcal{G}^n)_{n\in\mathbb{N}}:\mathcal{P}(\T^d)\rightarrow \R$ of uniformly Lipschitz\footnote{The Lipschitz constant of $\mathcal{G}^n$ is the same as that of $\mathcal{G}$.} and $C^{2,2}(\mathcal{P}(\T^d))$ functions such that $\mathcal{G}^n\xrightarrow{n\rightarrow \infty}\mathcal{G}$ uniformly. For any $\delta>0$ there exists $n_0\in\mathbb{N}$ such that $\mathcal{G}(\mu)\leq \mathcal{G}^n(\mu)+\delta$ for any $n\geq n_0$ and $\mu\in\mathcal{P}(\T^d)$. The regularity of $\mathcal{G}^n$ implies that there exists a constant $C_n>0$ such that $V^{n,\delta}(t,m):=\mathcal{G}^n(\pi^1_*m)+C_n(T-t)+\delta$ is a classical supersolution of \eqref{HJ1}, where $n\geq n_0$. Therefore, by comparison principle, we have $U^{\e}(t,m)\leq V^{n,\delta}(t,m)$. By taking the upper semi-limit in this inequality we get $\overline{U}(T,\mu)\leq \mathcal{G}^n(\mu)+\delta$ for any $n\geq n_0$. The result follows by sending $n\rightarrow \infty$ and $\delta\rightarrow 0$.
\vspace{2mm}


\noindent
For the second part we argue by contradiction. Assume that there exists a test function (Definitions \ref{test functions}, \ref{defz} and \ref{vdef}) $\varphi$ and a point $(\overline{t},\bar{z},\overline{\mu})\in [0,T)\times\T^d\times \mathcal{P}(\mathbb{T}^d)$ such that
$\widetilde{\bar{U}}(t,z,\mu)-\varphi(t,z,\mu)$ has a strict maximum with value $0$ (recall Remark \ref{strict}) at $(\overline{t},\bar{z},\overline{\mu})$
and 
\be \label{fi}
\begin{split} 
-\partial_t \varphi(\overline{t},\bar{z},&\overline{\mu}) +\int_{\mathbb{T}^d}\overline{H}\left(x+\overline{z},D_{\mu}\varphi(\overline{t},\overline{z},\overline{\mu},x),\overline{\mu}\right)\overline{\mu}(dx)\\
&-\int_{\mathbb{T}^d}\text{tr}\left(\sigma_1\sigma_1^{\top}D_xD_{\mu}\varphi(\overline{t},\overline{\mu},\overline{\mu},x)\right)\overline{\mu}(dx)-\text{tr}\left(\sigma_{0,1}\sigma_{0,1}^\top D^2_{zz}\varphi (\bar{t},\bar{z},\bar{\mu})\right)\geq 10\eta.
\end{split}
\ee
for some $\eta>0$. We will assume that $\overline{t}>0$; the case $\overline{t}=0$ can be treated similarly. Note also that, due to Lemma \ref{regularization} and \eqref{regofeff}, we may assume that $\varphi$ is such that $D^2_{xx}D_m\varphi$ and $D_xD_m\varphi$ exist and are continuous with \eqref{fi} still holding (perhaps for a different $\eta$) and $\widetilde{\bar{U}}(t,z,\mu)-\varphi(t,z,\mu)$ having a strict maximum with value $0$ at $(\overline{t},\overline{z},\overline{\mu})$ (perhaps for a different $(\overline{t},\overline{z},\overline{\mu})$). By Proposition \ref{lemma1} after setting $x\mapsto x+\overline{z}$ and $p\mapsto D_{\mu}\varphi(\overline{t},\overline{z},\overline{\mu},x)$, for every $\delta>0$ and $x,z\in \T^d$ there exists a function $v_{x,z}^{\delta}$ such that
\begin{align} 
\delta v^{\delta}-\text{tr}\left((\sigma_2\sigma_2^{\top}+\sigma_{0,2}\sigma_{0,2}^{\top})D^2_{yy} v^{\delta}\right)+H\left(x+\overline{z},y+z,D_{\mu}\varphi(\overline{t},\overline{z},\overline{\mu},x),D_yv^{\delta},\overline{\mu}\right)=0, \label{vdelta}
\end{align}
for all $y\in \mathbb{T}^d$. It is clear that $v^{\delta}$ has only two arguments $v^{\delta}_{x,z}(\cdot)=v^{\delta}(x,\cdot+z)$ and, furthermore, $\delta$ can be chosen small enough such that
\be \label{est1}
\left\| \delta v^{\delta}(x,z+\cdot)+\overline{H}\left(x+\overline{z},D_{\mu}\varphi(\overline{t},\overline{z},\overline{\mu},x),\overline{\mu}\right)\right\|_{L^{\infty}}\leq \eta,\;\; \text{for any }x,z\in \mathbb{T}^d.
\ee
We now consider the function $\psi^{\e,\delta}:[0,T]\times\T^d\times\T^d\times\mathcal{P}(\T^{2d})$ with $\psi^{\e,\delta}(t,z_1,z_2,m):=\varphi(t,z_1,\pi^1_*m)+\e \int_{\mathbb{T}^{2d}}v^{\delta}(x,y+z_2)m(dx,dy)$. For $\mu\in\mathcal{P}(\T^d)$, we have by the uniform convergence of $\psi^{\e,\delta}$ to $\varphi$ (as $\e\rightarrow 0$)
\begin{align*}
\limsup_{\substack{\e\rightarrow 0 \\ t'\rightarrow t\\ \mu'\rightarrow \mu} }\sup_{\substack{m\in \mathcal{P}(\mathbb{T}^{2d}):\\
\pi^1_*m=\mu'}} (\widetilde{U}^{\e}-\psi^{\e,\delta})(t',z_1,z_2,m)
&=\limsup_{\substack{\e\rightarrow 0 \\ t'\rightarrow t\\ \mu'\rightarrow \mu} }\sup_{\substack{m\in \mathcal{P}(\mathbb{T}^{2d}):\\
\pi^1_*m=\mu'}}\widetilde{U}^{\e}(t',z_1,z_2,m)-\varphi(t,z_1,\mu)\\
&=\limsup_{\substack{\e\rightarrow 0 \\ t'\rightarrow t\\ \mu'\rightarrow (\text{Id}+z_1)_*\mu} }\sup_{\substack{m\in \mathcal{P}(\mathbb{T}^{2d}):\\
\pi^1_*m=\mu'}}U^\e(t',m)-\varphi(t,z_1,\mu)\\
&=\bar{U}(t,(\text{Id}+z_1)_*\mu)-\varphi(t,z_1,\mu)\\
&=\widetilde{\bar{U}}(t,z_1,\mu)-\varphi(t,z_1,\mu).
\end{align*}
Since $(\overline{t},\overline{z},\overline{\mu})$ is a position of strict maximum for $\widetilde{\overline{U}}-\varphi$, there exists an $r>0$ small enough such that the above limit is negative for every $(t,z_1,\mu)\in (\overline{t}-r,\overline{t}+r)\times\{z: {\bf d}_{\T^d}(z,\overline{z})\le r\}\times  B_r(\mu)$, where $B_r(\mu)=\{ \mu\in \mathcal{P}(\mathbb{T}^d): {\bf d}_1(\mu,\overline{\mu})\le r\}$ and ${\bf d}_{\T^d}$ is the distance between points in $\T^d$.
However, by compactness, this implies that there exists an $\tilde{\eta}>0$ such that 
$\widetilde{U}^{\e}(t,z_1,z_2,m)-\psi^{\e,\delta}(t,z_1,z_2,m)\leq -\tilde{\eta}$ for every $(t,z_1,z_2,m)$ on the boundary of the set 
$$A_r:=(\overline{t}-r,\overline{t}+r)\times\{(z_1,z_2)\in\T^{2d}: {\bf d}_{\T^d}(z_1,\overline{z})\le r\}\times \left\{ m\in \mathcal{P}(\mathbb{T}^d): {\bf d}_1(\pi^1_*m,\overline{\mu})<r\right\}.$$ 
We claim that for $\e$ small enough
\be \label{claim}
\widetilde{U}^{\e}(t,z_1,z_2,m)-\psi^{\e,\delta}(t,z_1,z_2,m)\leq -\tilde{\eta},\;\;\text{in }A_r.
\ee 
\vspace{0.5mm}

\noindent
Before proving \eqref{claim}, we explain why it gives a contradiction. Indeed, if \eqref{claim} holds, then taking the semi-limit on both sides yields
$$\widetilde{\bar{U}}(t,z_1,\mu)-\varphi(t,z_1,\mu)\leq -\tilde{\eta},$$
for every $t\in (\overline{t}-r,\overline{t}+r)$, $z_1\in \{ z\in \T^d: {\bf d}_{\T^d}(z,\overline{z})\le r\}$ and $\mu\in B_r$. In particular, setting $t=\overline{t}$, $z_1=\overline{z}$ and $\mu=\overline{\mu}$, we get a contradiction.
\vspace{2mm}

\noindent
It remains to show \eqref{claim}. If \eqref{claim} is not true, we may suppose that 
$$\max_{\overline{A_r}}\bigg\{\widetilde{U}^{\e}(t,z_1,z_2,m)-\psi^{\e,\delta}(t,z_1,z_2,,m)\bigg\}=\widetilde{U}^{\e}(t_0,z_{0,1},z_{0,2},m_0)-\psi^{\e,\delta}(t_0,z_{0,1},z_{0,2},m_0),$$
where $(t_0,z_{0,1},z_{0,2},m_0)$ is an interior point of $A_r$.
We observe that, due to the regularity of $v^{\delta}$ proved in Proposition \ref{lemma1}, $\psi^{\e,\delta}$ is a test function for \eqref{HJ1} (see Definition \ref{test functions}), hence by the definition of viscosity solutions
\begin{align}\label{calc}
&-\partial_t\varphi(t_0,z_{0,1},\pi^1_*m_0)\nonumber\\
&+\int_{\mathbb{T}^{2d}}H\left(x+z_{0,1},y+z_{0,2},D_{\mu}\varphi(t_0,z_{0,1},\pi^1_*m_0,x)+\e D_xv^{\delta}(x,y+z_{0,2}),D_yv^{\delta}(x,y+z_{0,2}),\pi^1_*m_0\right)m_0(dx,dy)\nonumber \\
&
\hspace{0.5cm}-\int_{\mathbb{T}^{2d}}\text{tr}\left(\sigma_1\sigma_1^{\top}D_xD_{\mu}\varphi(t_0,z_{0,1},\pi^1_*m_0,x)\right)\pi^1_*m_0(dx)\nonumber\\
&\hspace{1cm}-\e\int_{\mathbb{T}^{2d}}\text{tr}\left(\sigma_1\sigma_1^{\top}D^2_{xx}v^{\delta}(x,y+z_{0,2})\right)m_0(dx,dy)\nonumber\\
&\hspace{1.5cm}-\int_{\mathbb{T}^{2d}}\text{tr}\left(\sigma_2\sigma_2^{\top}D^2_{yy}v^{\delta}(x,y+z_{0,2})\right)m_0(dx,dy)\nonumber\\
&\hspace{0.8cm}-\text{tr}\left(\sigma_{0,1}\sigma_{0,1}^{\top}D^2_{z_1z_1}\varphi(t_0,z_{0,1},\pi^1_*m_0)\right)-\int_{\T^{2d}}\text{tr}\left( \sigma_{0,2}\sigma_{0,2}^{\top}D^2_{z_2z_2}v^{\delta}(x,y+z_{0,2})\right)m_0(dx,dy)\leq 0.
\end{align}
Since $r$ is small enough and $(t_0,z_{0,1},z_{0,2},m_0)\in A_r$, by the uniform continuity of $\partial_t\varphi, D^2_{zz}\varphi$ and $D_xD_{\mu}\varphi$ and the Lipschitz continuity of $D_xD_{\mu}\varphi(t,z_1,\mu,\cdot)$, we may assume that
\be\label{est2}
|\partial_t\varphi(t_0,z_{0,1},\pi^1_*m_0)-\partial_t\varphi(\overline{t},\overline{z},\overline{\mu})|\leq \tilde{C}\eta,
\ee 
\be \label{est3}
\bigg| \int_{\mathbb{T}^d}\text{tr}\left(\sigma_1\sigma_1^{\top}D_xD_{\mu}\varphi(t_0,z_{0,1},\pi^1_*m_0,x)\right)\pi^1_*m_0(dx)-\int_{\mathbb{T}^d}\text{tr}\left(\sigma_1\sigma_1^{\top}D_xD_{\mu}\varphi(\overline{t},\overline{z},\overline{\mu},x)\right)\overline{\mu}(dx)\bigg|\leq \tilde{C}\eta,
\ee 
\be \label{est98}
\left| \text{tr}\left( \sigma_{0,1}\sigma_{0,1}^{\top}D^2_{zz}\varphi(t_0,z_{0,1},\pi^1_*m_0) \right) -\text{tr}\left( \sigma_{0,1}\sigma_{0,1}^{\top}D^2_{zz}\varphi(\overline{t},\overline{z},\overline{\mu}) \right)\right|\leq \tilde{C}\eta,
\ee
for some constant $\tilde{C}$ small enough. In addition, by the properties of $H$ (assumption (A3), equation \eqref{A2'}) and the uniform continuity of $D_{\mu}\varphi$ we have
\begin{align}
\bigg|H&\left(x+z_{0,1},y+z_{0,2},D_{\mu}\varphi(t_0,z_{0,1},\pi^1_*m_0,x)+\e D_xv^{\delta}(x,y+z_{0,2}),D_yv^{\delta}(x,y+z_{0,2}),\pi^1_*m_0\right)\nonumber\\
&\hspace{5.5cm}-H\left( x+\overline{z},y+z_{0,2},D_{\mu}\varphi(\overline{t},\overline{z},\overline{\mu},x),D_yv^{\delta}(x,y+z_{0,2}),\overline{\mu} \right)\bigg|\nonumber\\
&\leq C_H\left(1+\|D_yv^{\delta}\|_{\infty}+\e \|D_xv^{\delta}\|_{\infty}+\|D_{\mu}\varphi\|_{\infty}\right)\nonumber\\
&\hspace{2cm}\cdot\left(|D_{\mu}\varphi(t_0,z_{0,1},\pi^1_*m_0,x)-D_{\mu}\varphi(\overline{t},\overline{z},\overline{\mu},x)|+|z_{0,1}-\overline{z}|+\e\|D_xv^{\delta}\|_{\infty}+{\bf d}_1(\pi^1_*m_0,\overline{\mu})\right)\nonumber\\
&\leq C\left(|D_{\mu}\varphi(t_0,z_{0,1},\pi^1_*m_0,x)-D_{\mu}\varphi(\overline{t},\overline{z},\overline{\mu},x)|+r+\e+r\right)\leq \eta +C\e\label{est5},
\end{align}
for some constant $C$ and $r$ small enough. We combine \eqref{est2}, \eqref{est3}, \eqref{est98}, 
\eqref{est5} with \eqref{calc}
to discover
\begin{align}\label{est6}
-\partial_t\varphi(\overline{t},&\overline{z},\overline{\mu})+\int_{\mathbb{T}^{2d}}H\left(x+\overline{z},y+z_{0,2},D_{\mu}\varphi(\overline{t},\overline{z},\overline{\mu},x),D_yv^{\delta}(x,y+z_{0,2}),\overline{\mu}\right)m_0(dx,dy)\nonumber\\
&-\int_{\mathbb{T}^{2d}}\text{tr}\left(\sigma_1\sigma_1^{\top}D_xD_{\mu}\varphi(\overline{t},\overline{z},\overline{\mu},x)\right)\overline{\mu}(dx)\nonumber \\
&-\e\int_{\mathbb{T}^{2d}}\text{tr}\left(\sigma_1\sigma_1^{\top}D^2_{xx}v^{\delta}(x,y+z_{0,2})\right)m_0(dx,dy)\nonumber\\
&-\int_{\mathbb{T}^{2d}}\text{tr}\left(\sigma_2\sigma_2^{\top}D^2_{yy}v^{\delta}(x,y+z_{0,2})\right)m_0(dx,dy)\nonumber\\
&-\text{tr}\left( \sigma_{0,1}\sigma_{0,1}^{\top}D^2_{zz}\varphi(\overline{t},\overline{z},\overline{\mu}) \right)-\int_{\T^{2d}}\text{tr}\left( \sigma_{0,2}\sigma_{0,2}^{\top}D^2_{z_2z_2}v^{\delta}(x,y+z_{0,2})  \right)m_0(dx,dy)\leq 4\eta +C\e.
\end{align}
Now, using the fact that $D^2_{zz}v^{\delta}(x,y+z)=D^2_{yy}v^{\delta}(x,y+z)$ combined with \eqref{vdelta} and the continuity of $D^2_{xx}v^{\delta}$ provided by Proposition \ref{lemma1}, \eqref{est6} becomes
\be \label{est7}
\begin{split}
-\partial_t\varphi&(\overline{t},\overline{z},\overline{\mu})-\int_{\mathbb{T}^{2d}} \delta v^{\delta}(x,y+z_{0,2})  m_0(dx,dy)-\e C_1 \\
&-\int_{\mathbb{T}^{2d}}\text{tr}\left(\sigma_1\sigma_1^{\top}D_xD_{\mu}\varphi(\overline{t},\overline{z},\overline{\mu},x)\right)\overline{\mu}(dx)-\text{tr}\left(\sigma_{0,1}\sigma_{0,1}^{\top}D^2_{zz}\varphi(\overline{t},\overline{z},\overline{\mu})\right)  \leq 4\eta +C\e,
\end{split}
\ee
for some constant $C_1$. However, with the help of estimate \eqref{est1}, this implies
\be \label{est8}
\begin{split}
-\partial_t&\varphi(\overline{t},\overline{z},\overline{\mu})+\int_{\mathbb{T}^{d}}\overline{H}(x+\overline{z},D_{\mu}\varphi(\overline{t},\overline{z},\overline{\mu},x),\overline{\mu})\pi^1_*m_0(dx)\\
&-\int_{\mathbb{T}^{2d}}\text{tr}\left(\sigma_1\sigma_1^{\top}D_xD_{\mu}\varphi(\overline{t},\overline{\mu},x)\right)\overline{\mu}(dx)-\text{tr}\left(\sigma_{0,1}\sigma_{0,1}^{\top}D^2_{zz}\varphi (\overline{t},\overline{z},\overline{\mu}) \right) \leq 3\eta +\e (C_1+C).
\end{split}
\ee
Finally, due to \eqref{regofeff} and the regularity $\varphi$, $\overline{H}\left(x+\overline{z},D_{\mu}\varphi(\bar{t},\overline{z},\bar{\mu},x),\overline{\mu}\right)$ is Lipschitz in $x$, so for $r$ small enough we can write
$$\bigg|\int_{\mathbb{T}^{d}}\overline{H}\left(x+\overline{z},D_{\mu}\varphi(\overline{t},\overline{z},\overline{\mu},x)\right)(\pi^1_*m_0-\overline{\mu})(dx)\bigg|\leq C'{\bf d}_1(\pi^1_*m_0,\overline{\mu})\leq C'r\leq \eta.$$
Since $C,C_1,$ depend only on $H,\varphi$ and the $C^{2}$-norm of $v^{\delta}$ (which is independent of $\delta$, by our arguments in Proposition \eqref{lemma1}) we may also assume that $\e$ is small enough such that $\e (C+C_1)\leq \eta$, hence \eqref{est8} becomes
\begin{align*} 
-\partial_t \varphi(\overline{t},\bar{z},&\overline{\mu}) +\int_{\mathbb{T}^d}\overline{H}\left(x+\overline{z},D_{\mu}\varphi(\overline{t},\overline{z},\overline{\mu},x),\overline{\mu}\right)\overline{\mu}(dx)\nonumber\\
&-\int_{\mathbb{T}^d}\text{tr}\left(\sigma_1\sigma_1^{\top}D_xD_{\mu}\varphi(\overline{t},\overline{\mu},\overline{\mu},x)\right)\overline{\mu}(dx)-\text{tr}\left(\sigma_{0,1}\sigma_{0,1}^\top D^2_{zz}\varphi (\bar{t},\bar{z},\bar{\mu})\right)\leq 4\eta.
\end{align*}
This contradicts \eqref{fi}.
\end{proof}


\vspace{3mm}

\subsection{Proof of Corollary \ref{linear case}}
We finally prove Corollary \ref{linear case}.

\begin{proof}
Under our assumptions, $U^{\e}(t,m)= \mathcal{G}\left(\mathcal{L}(X_T^{\e,t,m})\right)$ is the unique viscosity solution of \eqref{HJ1}. In addition, for $x\in\T^d,p\in \R^d$ and $\mu=\pi^1_*m\in\mathcal{P}(\T^d)$, the effective Hamiltonian $\overline{H}(x,p,\mu)$ derived in Proposition \ref{lemma1} is such that the following equation has a unique solution up to a constant
\be\label{eff lin case}
-\text{tr}\left((\sigma_2\sigma_2^{\top}+\sigma_{0,2}\sigma_{0,2}^{\top})D^2_{yy}v\right)-p\cdot c(x,y,\mu)-f(x,y,\mu)\cdot D_yv- \overline{H}(x,p,\mu)=0,\;\;y\in \T^d.
\ee
It was proved in \cite{bensoussan2011asymptotic} (see also \cite{bezemek2023rate,zitridis2023homogenization}) that $\overline{H}(x,p,\mu)= \int_{\T^d} p\cdot c(x,y,\mu) \pi_{x,\mu}(dy)$, where $\pi_{x,\mu}$ is the invariant measure associated to the problem $-\text{tr}(\sigma_2\sigma_2^{\top}D^2_{yy}v)-f(x,y,\mu)\cdot D_yv=0$. We set $\overline{c}(x,\mu):=\int_{\T^d}c(x,y,\mu)\pi_{x,\mu}(dy)$, where the integration is coordinate by coordinate. By Theorem \ref{main} we have that the limit of $U^{\e}$ is the unique viscosity solution of the linear Hamilton-Jacobi equation posed in $[t_0,T]\times \mathcal{P}(\T^d)$
$$\begin{cases}
    -\partial_t U+\int_{\T^d}\overline{c}(x,\mu)\cdot D_{\mu}U(t,\mu,x)\mu(dx)-\int_{\T^d}\text{tr}\left( (\sigma_1\sigma_1^{\top}+\sigma_{0,1}\sigma_{0,1}^{\top})D_xD_{\mu}U(t,\mu,x)\right)\mu(dx)\\
    \hspace{4cm}-\int_{\T^d}\int_{\T^d}\text{tr}\left(\sigma_{0,1}\sigma_{0,1}^{\top}D^2_{\mu\mu}U(t,\mu,x,x') \right)\mu(dx)\mu(dx')=0,\\
    U(T,\mu)=\mathcal{G}(\mu).
\end{cases}$$
This is the Hamilton-Jacobi equation associated with \eqref{SDE cor}, hence following a standard dynamic programming principle argument (see \cite[Proposition 6.3]{daudin2023well}),
we deduce that $U(t,m)=\mathcal{G}\left(\mathcal{L}(X_T^{t,\pi^1_*m})\right)$ is its viscosity solution, which is also unique (Theorem \ref{joe}). The proof is complete.
\end{proof}

\appendix

\section{Wasserstein space and Wasserstein Derivatives}\label{appendix A}

\noindent
In this section, for the reader's convenience, we will gather the definitions of Wasserstein spaces and Wasserstein derivatives.

\begin{definition}\label{wass}
Let $n_0\in\mathbb{N}$ and $p\geq 1$. We define the $p-$Wasserstein space as 
$$\mathcal{P}_p(\T^{n_0})=\bigg\{\mu \text{ probability measure over }\T^{n_0} \bigg| \text{M}_p(\mu):=\int_{\T^{n_0}}|x|^p\mu(dx)<+\infty  \bigg\}$$
with metric ${\bf d}_p(\mu,\nu)=\inf_{\pi\in \Pi(\mu,\nu)}\bigg( \int_{\T^{2n_0}}|x-y|^p \pi(dxdy)\bigg)^{1/p}$.
\end{definition}
\vspace{1mm}

\begin{remark}\label{remarkappend}
(i) It is known that for any $p\geq 1$, the metrics ${\bf d}_p$ are equivalent.\\
(ii) The $p-$Wasserstein space $\mathcal{P}_p(\T^{n_0})$ is compact.\\
(iii) Let $p>1$. There exists $C:=C(p)>0$ such that for any $\mu_1,\mu_2\in \mathcal{P}_p(\T^{n_0})$ we have ${\bf d}_p(\mu_1,\mu_2)\geq {\bf d}_1(\mu_1,\mu_2)\geq C{\bf d}_p(\mu_1,\mu_2)^p$.\\
(iv) For $p=1$ we also have the formula ${\bf d}_1(\mu,\nu)=\sup_{f\in \text{1-Lip}}\int_{\T^{n_0}}f(x)(\mu-\nu)(dx)$.
\end{remark}
\vspace{1mm}

\begin{definition}
Let $U:\mathcal{P}_p(\T^{n_0})\rightarrow \R$ be a function for some $p\geq 1$. We say that $U$ is continuously differentiable or that it is $C^1$ in $\mathcal{P}_p(\T^{n_0})$ if there exists a continuous function $\frac{\delta U}{\delta m}:\mathcal{P}_p(\T^{n_0})\times \T^{n_0}\rightarrow \R$ such that for every $m_1,m_2\in\mathcal{P}_p(\T^{n_0})$ we have
$$\lim_{t\rightarrow 0}\frac{U(tm_1+(1-t)m_2)-U(m_2)}{t}=\int_{\T^{n_0}}\frac{\delta U}{\delta m}(m_2,x)(m_2-m_2)(dx).$$
If the $x$-derivative exists, we use the notation $D_mU(m,x):=D_x\frac{\delta U}{\delta m}(m,x)$, for $x\in \T^{n_0}$ and $m\in\mathcal{P}_p(\T^{n_0})$. $D_mU$ is called the Wasserstein derivative of $U$.
\end{definition}

\noindent
We can similarly define higher order Wasserstein derivatives. We prove the following proposition, which is taken from \cite[Proposition A.5]{zitridis2023homogenization}.
\begin{proposition}\label{Wass reg}
Let $U:\mathcal{P}_p(\T^{n_0})\rightarrow \mathbb{R}$ be a $C^1$ function and $\pi^i:\T^{2n_0}\rightarrow \T^{n_0}$ such that $\pi^i(x_1,x_2)=x_i$, $i=1,2$. Then, for $i=1,2$ the function $\tilde{U}_i:\mathcal{P}_p(\T^{2n_0})\rightarrow \R$ with $\tilde{U}_i(m)=U(\pi^i_*m)$ is also $C^1$ with 
\be \label{regularity 2}
\frac{\delta \tilde{U}_i}{\delta m}(m,x_1,x_2)=\frac{\delta U}{\delta \mu}(\pi^i_*m,\pi^i(x_1,x_2)),\;\text{ for every }m\in\mathcal{P}_p(\T^{2n_0}),\; x_1,x_2\in\T^{n_0}.
\ee
In particular, $D_m\tilde{U}_i(m,x_1,x_2)=D_{\mu}U(\pi^i_*m,x_i)$ for every $m\in\mathcal{P}_p(\T^{2n_0})$ and $x_1,x_2\in \T^{n_0}$, $i=1,2$.
\end{proposition}

\begin{proof}
We will show the result for $i=1$ as the other case can be shown with an identical argument. For simplicity in the calculations below we drop the index $i$, so $\tilde{U}_i=\tilde{U}$ and $\pi^1=\pi$. Let $m_1,m_2\in\mathcal{P}_p(\T^{2n_0})$ and $\pi_*m_1=\mu_1,\; \pi_*m=\mu_2$. We have by the properties of the pushforward
\begin{align*}
\lim_{t\to 0} \frac{ \tilde{U}((1-t)m_1+tm_2)-\tilde{U}(m_1)}{t}&=\lim_{t\to 0} \frac{ U((1-t)\mu_1+t\mu_2)-U(\mu_2)}{t}\\
&=\int_{\T^{n_0}} \frac{\delta U}{\delta \mu}(\mu_1,x_1)(\mu_2-\mu_1)(dx_1)\\
&=\int_{\T^{n_0}} \frac{\delta U}{\delta \mu}(\pi_*m_1,x_1)(\pi_*m_2-\pi_*m_1)(dx_1,dx_2)\\
&=\int_{\T^{n_0}} \frac{\delta U}{\delta \mu}(\pi_*m_1,x_1)(\pi_*m_2-\pi_*m_1)(dx_1)\\
&= \int\int_{\T^{n_0}\times\T^{n_0}} \frac{\delta U}{\delta \mu}(\pi_*m_1,\pi(x_1,x_2))(m_2-m_1)(dx_1,dx_2).
\end{align*}
Equation (\ref{regularity 2}) follows. The regularity of $\tilde{U}$ is a direct implication of \eqref{regularity 2} and the fact that $U$ is $C^1$.
\end{proof}

\section{Technical proofs}
\noindent
In this section we provide the proofs of some Propositions which were too long and would decrease the readability of the paper.

\begin{lemma}\label{parametric}
    Let $\delta>0$, $\sigma_2$ an invertible $d\times d$ matrix and $f:\R^d\times \T^d\rightarrow \R^d$, $g:\R^d\times \T^d\rightarrow \R$ two $C^1$ functions. Suppose that $u:= u(x,y): \R^d\times \T^d$ satisfies the elliptic equation 
    \be\label{elliptic}
\delta u-\text{tr}\left(\sigma_2\sigma_2^{\top}D^2_{yy}u\right) +f(x,y)\cdot D_y u= g(x,y),
    \ee
    in the classical sense. Then, $u$ is $C^1$; that is $D_xu(x,y)$ and $D_yu(x,y)$ exist and are continuous.
\end{lemma}

\begin{proof}
Fix $x\in \R^d$ and $B\subset \R^d$ a closed ball centered at $x$. In the proof we will be working in $B\times \T^d$, which is compact. Let $h\in \R^d$ such that $x+h\in B$, $\xi=\frac{h}{|h|}$ and $\tilde{u}_h(x,y)=u(x+h,y)-u(x,y)$. The function $\tilde{u}_h$ satisfies
\be \label{elliptic2}
\begin{split}
\delta \tilde{u}_h-&\text{tr}\left( \sigma_2\sigma_2^{\top}D^2_{yy}\tilde{u}_h\right) +f(x+h,y)\cdot D_y\tilde{u}_h\\
&=g(x+h,y)-g(x,y)-(f(x+h,y)-f(x,y))\cdot D_yu(x,y).
\end{split}
\ee
By the compactness of $B\times \T^d$, we have $\|D_yu(x,\cdot)\|_{L^{\infty}}\leq C$ and $|f(x+h,y)|\leq \sup_{(z,y)\in B\times \T^d}|f(z,y)|$, for some constant $C>0$. In addition, for any $y\in \T^d$
\begin{align*}
|g(x+h,y)-g(x,y)|&\leq |h|\int_0^1 |D_xg(sh+x,y)|ds\leq |h|\sup_{(z,y)\in B\times \T^d}|D_xg(z,y)|,\\
|f(x+h,y)-f(x,y)|&\leq |h|\sup_{(z,y)\in B\times \T^d}|D_xf(z,y)|.
\end{align*}
By virtue of these bounds and the maximum principle, we have that there exists a constant $C>0$ depending only on $\sigma_2,\delta, d, f,g$ and $B$ such that $\|\tilde{u}_h(x,\cdot)\|_{L^{\infty}}\leq C$. Now, standard H\"older estimates for elliptic equations (see \cite[Chapter 8]{gilbarg1977elliptic}), yield 
\begin{align}
\|\tilde{u}_h\|_{C^{1,\alpha}}\leq C( \|\tilde{u}_h\|_{L^{\infty}}+C'|h|)\leq C|h|,\label{estimate2}
\end{align}
for some $\alpha\in (0,1)$ and where $C$ depends only on $\delta,d, \sigma_2, f,g$ and $B$. On the on hand \eqref{estimate2} implies that $D_yu(x,y)$ is continuous. On the other hand, it allows us to use Arzela-Ascoli's theorem to get a uniformly convergent subsequence of $\tilde{u}_h(x,\cdot)/|h|$ in $C^1$ as $|h|\rightarrow 0$. Let $\bar{u}_{\xi}$ be the limit, which is continuous due to the uniform convergence. Since $\tilde{u}_h$ is a classical solution of \eqref{elliptic2}, it is also a weak solution, hence $\bar{u}_{\xi}$ is a weak solution of
\begin{align}
\delta \bar{u}_{\xi}-\text{tr}\left(\sigma_2\sigma_2^{\top}D^2_{yy}\bar{u}_{\xi}\right)+f(x,y)\cdot D_y\bar{u}_{\xi}=D_xg(x,y)\cdot \xi-D_xf(x,y)\xi\cdot D_yu(x,y).\label{xi}
\end{align}
Since the weak solution is unique, we deduce that $\frac{\tilde{u}_h(x,\cdot)}{|h|} \xrightarrow{h\rightarrow 0} \bar{u}_{\xi}$, therefore the directional derivatives of $u(x)$ exist and hence $D_xu(x,y)$ exists. From \eqref{estimate2} by sending $|h|\rightarrow 0$, it also follows that $D_xu(x,\cdot)\in C^{1,\alpha}(\T^d)$ with a H\"older constant which is uniform in $B$ (since $x\in B)$.\\
To prove the continuity of $D_xu(x,y)$, since $D_xu(x,\cdot)\in C^{1,\alpha}(\T^d)$ with a H\"older constant which is uniform in $B$, it suffices to show that $x\mapsto D_xu(x,y)$ is continuous in $B$. We will show that its components $u_{x_i}(x,y)$, $i=1,...,d$, are continuous with respect to $x$. Indeed, let $x\in B$ and $h\in \R^d$ such that $x+h\in B$. We consider the difference $\bar{w}^h_i(x,y):= u_{x_i}(x+h,y)-u_{x_i}(x,y)$. By choosing the appropriate $\xi= (0,...,0,1,0,...,0)$, where $1$ is in the $i$-th position in \eqref{xi}, we discover that $\bar{w}^h_i$ satisfies
\be \label{difference}
\begin{split}
& \delta \bar{w}^h_i(y)
- \text{tr}\!\left( \sigma_2\sigma_2^{\top} D^2_{yy} \bar{w}^h_i(y) \right)
+ f(x+h,y) \cdot D_y \bar{w}^h_i(y) \\
& \quad = g_{x_i}(x+h,y) - g_{x_i}(x,y) \\
& \quad\ \ \ - f_{x_i}(x+h,y) \cdot D_y u(x+h,y)
+ f_{x_i}(x,y) \cdot D_y u(x,y) \\
& \quad\ \ \ - \big( f(x+h,y) - f(x,y) \big) \cdot D_y u_{x_i}(x,y)
\end{split}
\ee
in the weak sense. Since $D_xu(x,\cdot)\in C^{1,\alpha}(\T^d)$ and $f,g\in C^1$, we observe that the coefficients on the left hand side and the right hand side of \eqref{difference} are uniformly bounded (with respect to $h$) as functions of $y$, hence again by standard H\"older estimates $\overline{w}_i^h$ is uniformly bounded in $C^{1,\alpha}$. By Arzela-Ascoli, $\overline{w}_i^h$ converges uniformly (up to a subsequence) as $|h|\rightarrow 0$ to a function $\overline{w}_i$. Since the right hand side of \eqref{difference} converges to $0$ unifromly as $|h|\rightarrow 0$, $\overline{w}$ satisfies weakly the equation
$$\delta \overline{w}_i-\text{tr}\left(\sigma_2\sigma_2^{\top}D^2_{yy}\overline{w}_i\right)+f(x,y)\cdot D_y\overline{w}_i=0.$$
\vspace{-2mm}

\noindent
We notice that $\overline{w}_i=0$ is the unique weak solution to this equation, hence $\overline{w}_i^{h}\xrightarrow{|h|\rightarrow 0}0$ uniformly. We conclude that $\lim_{|h|\rightarrow 0} \| u_{x_i}(x+h,\cdot)-u_{x_i}(x,\cdot)\|_{L^{\infty}}=0$, which gives us the desired continuity.
\end{proof}

\vspace{2mm}

\noindent
We, finally, prove a regularization result for functions defined on the Wasserstein space $\mathcal{P}_p(\T^d)$, $p \geq 1$. The idea comes from \cite[Chapter 5]{carmona2018probabilistic}.
\begin{lemma}\label{regularization}
Let $\Phi:\mathcal{P}_p(\T^d)\rightarrow \R$ be a function, which is Lipschitz and continuously differentiable. For a smooth symmetric density on $\R^d$ with compact support and $\delta>0$, let
$$\rho_{\delta}(x):=\frac{1}{\delta^d}\rho\left(\frac{x}{\delta}\right),\quad x\in \R^d$$
and $\Phi^{\delta}(m):=\Phi(m \ast \rho_{\delta})$, $m\in\mathcal{P}_p(\T^d)$. Then, $\Phi^{\delta}$ is continuously differentiable with
$$\frac{\delta\Phi^{\delta}}{\delta m}(m,x)=\left( \frac{\delta\Phi}{\delta m}(m\ast\rho_{\delta},\cdot)\ast \rho_{\delta}\right)(x),\quad x\in \T^d,\; m\in\mathcal{P}_p(\T^d).$$
Moreover, if $\frac{\delta\Phi}{\delta m}(m,\cdot)\in C^2$, then there exist a positive constant $C:=C(d,p,\rho)$ and an increasing function $\omega: [0,+\infty)\rightarrow [0,+\infty)$ with $\lim_{t\rightarrow 0^+}\omega(t)=\omega(0)=0$ such that
\be\label{est101}
\sup_{m\in\mathcal{P}_p(\T^d)}|\Phi^{\delta}(m)-\Phi(m)|+\sup_{m\in\mathcal{P}_p(\T^d)}\left\| \frac{\delta\Phi^{\delta}}{\delta m}(m,\cdot)-\frac{\delta\Phi}{\delta m}(m,\cdot)\right\|_{C^2(\T^d)}\leq C\omega(\delta)
\ee
\end{lemma}

\begin{proof}
The proof for the derivative can be found in \cite[Chapter 5]{carmona2018probabilistic} and the proof for the first term of \eqref{est101} follows from the Lipschitz continuity of $\Phi$ as in \cite[Lemma 4.2]{daudin2024optimal}. Indeed, for $m\in \mathcal{P}_p(\T^d)$ and $L_{\Phi}$ the Lipschitz constant of $\Phi$
\begin{align*}
|\Phi^{\delta}(m)-\Phi(m)|&\leq L_{\Phi} {\bf d}_p(m\ast\rho_{\delta},m)\leq L_{\Phi} {\bf d}_1(m\ast\rho_{\delta},m)^{1/p}\\
&\leq L_{\Phi}\left( \sup_{f\in\text{ 1-Lip}}\int _{\T^d}(f\ast\rho_{\delta}(x)-f(x))m(dx)\right)^{1/p}\leq L_{\Phi}\delta^{1/p}.
\end{align*}
For the second term of \eqref{est101} we argue similarly. Since the functions $D_xD_m\Phi(x,m)$ and $D_m\Phi(x,m)$ are continuous, by the compactness of $\T^d\times \mathcal{P}_p(\T^d)$, they must also be uniformly continuous, hence we may consider $\omega_{D_m\Phi}$ and $\omega_{D_xD_m\Phi}$ the corresponding moduli of continuity. We have for any $(m,x)\in \mathcal{P}_p(\T^d)\times \T^d$
\begin{align*}
|D_m\Phi(m\ast \rho_{\delta},\cdot)\ast \rho_{\delta}(x)-D_m\Phi(m,x)|\leq |D_m\Phi(m\ast \rho_{\delta},\cdot)&\ast \rho_{\delta}(x)-D_m\Phi(m,\cdot)\ast \rho_{\delta}(x)|\\
&+|D_m\Phi(m,\cdot)\ast \rho_{\delta}(x)-D_m\Phi(m,x)|.
\end{align*}
However,  
\begin{align*}
|D_m\Phi(m,\cdot)\ast \rho_{\delta}(x)-D_m\Phi(m,x)|&\leq \int_{\T^d}|D_m\Phi(m,x-y)-D_m\Phi(m,x)|\rho_{\delta}(y)dy\\
&\leq \int_{\T^d}\omega_{D_m\Phi}(|y|)\rho_{\delta}(y)dy\leq C\omega_{D_m\Phi}(\delta).
\end{align*}
and
\begin{align*}
|D_m\Phi(m\ast \rho_{\delta},\cdot)\ast \rho_{\delta}(x)-D_m\Phi(m,\cdot)\ast \rho_{\delta}(x)|&\leq \int_{\T^d} |D_m\Phi(m\ast \rho_{\delta},y)-D_m\Phi(m,y)|\rho_{\delta}(x-y)dy\\
&\leq \int_{\T^d} \omega_{D_m\Phi}({\bf d}_p(m\ast\rho_{\delta},m))\rho_{\delta}(x-y)dy\\
&\leq \int_{\T^d} \omega_{D_m\Phi}(\delta^{1/p})\rho_{\delta}(x-y)dy\\
&\leq C\omega_{D_m\Phi}(\delta^{1/p}).
\end{align*}
Thus, 
$$|D_m\Phi(m\ast \rho_{\delta},\cdot)\ast \rho_{\delta}(x)-D_m\Phi(m,x)|\leq C\max\left\{\omega_{D_m\Phi}(\delta^{1/p}),\; \omega_{D_m\Phi}(\delta)\right\}.$$
This shows the inequality for the first derivative. For the second derivative, we argue similarly but with the modulus $\omega_{D_xD_m\Phi}$. The proof is complete.
\end{proof}

\noindent
\textbf{Acknowledgements.} The author was partially supported by P. E. Souganidis’ NSF grant DMS-1900599, ONR grant N000141712095, and AFOSR grant FA9550-18-1-0494. The author wishes to thank Joe Jackson for insightful discussions regarding Theorem \ref{joe} and its proof.

\bibliographystyle{is-alpha}		
\bibliography{mfg}

\end{document}